%% file: main.tex
\documentclass{amsart}
\usepackage{amsmath}
\usepackage{amssymb}
\usepackage[dvipsnames]{xcolor}
\usepackage{amsthm}
\usepackage{thmtools}
\usepackage{thm-restate}
\usepackage{amsfonts}
\usepackage{graphicx}
\usepackage{stackrel}
\usepackage[style=alphabetic, backend = biber]{biblatex}
\usepackage{wrapfig}
\usepackage{pgfplots}
\usepackage{tikz}
\usepackage{tikz-cd}
\usepackage{upgreek}
\usepackage{caption}
\usepackage{subcaption}
\usepackage{relsize}
\usepackage{mathtools}
\usepackage{adjustbox}
\usepackage{tabularx}
\usepackage[shortlabels]{enumitem}
\usepackage[hyperindex,colorlinks,breaklinks]{hyperref}
\usepackage{hyperref}
\usepackage[hyphenbreaks]{breakurl}
\usepackage[mathscr]{eucal}

\newcommand{\bs}{\, \backslash \,}
\newcommand{\Gh}{\widehat{\Gamma}}
\newcommand{\Db}{\overline{\Delta}}
\newcommand{\Gg}{X}
\newcommand{\Gs}{\mathcal{G}}
\newcommand{\Hs}{\mathcal{H}}
\newcommand{\Ls}{\mathcal{L}}
\newcommand{\Zeta}{\upzeta}
\newcommand{\Chi}{\raisebox{0.175em}{$\upchi$}}

\newcommand{\ggx}{(\Gs,\Xi)}
\newcommand{\gxt}{(\Gs,\Xi,\Theta)}
\newcommand{\am}[1][\upchi]{\ast_{#1}}
\newcommand{\hgx}{(\Hs,\Upsilon)}
\newcommand{\hys}{(\Hs,\Upsilon,\Sigma)}

\newcommand{\lxt}{(\Ls,\Zeta, \Zeta_0)}
\newcommand{\PP}[1][\ggx]{\Pi_1{#1}}
\newcommand{\PT}[1][\gxt]{\Pi_1{#1}}

\newcommand{\PA}[1][\ggx]{\Pi_1^{\mathrm{abs}} #1}
\newcommand{\PAT}[1][\gxt]{\Pi_1^{\mathrm{abs}} #1}

\newcommand{\GU}{G / U}

\newcommand{\Ker}{\operatorname{Ker}}
\newcommand{\Stab}[2][]{\operatorname{Stab}_{#1}(#2)}

\newcommand{\Z}{\mathbb{Z}}
\newcommand{\Fp}{\mathbb{F}_p}

\newcommand{\pZ}{\widehat{\Z}}
\newcommand{\at}[1]{|_{#1}}

\newcommand{\Int}{\operatorname{Int}}

\newcommand{\wF}{\overline{\Phi}}
\newcommand{\wE}[1][\Xi]{\widetilde{E}(#1)}

\newcommand{\edge}[1][i]{ $\eta_{#1}$}
\newcommand{\vertexpower}[2][i]{ $\Gs(\tau_{#1})^{\delta(\eta_{#2})}$}
\newcommand{\vertexs}[1][i]{ $\Gs(\sigma_{#1})$}
\newcommand{\vertexpowers}[2][i]{ $l_{\eta_{#2}}^{-1}\Gs(\sigma_{#1})l_{\eta_{#2}}$}

\newcommand{\sge}[1][{t_e}]{\overline{\langle #1 \rangle}}

\newtheorem{theorem}{Theorem}
\numberwithin{theorem}{section}
\numberwithin{equation}{section}
\newtheorem{definition}{Definition}
\newtheorem{lemma}[theorem]{Lemma}

\newtheorem{proposition}[theorem]{Proposition}
\newtheorem{remark}[theorem]{Remark}
\newtheorem{corollary}[theorem]{Corollary}
\newtheorem{question}{Question}

\title[Profinite Amalgamated Factors]{Profinite Subgroup Accessibility and Recognition of Amalgamated Factors}
\author{Julian Wykowski}
\date{\today}
\address{Department of Pure Mathematics and Mathematical Statistics, Centre for Mathematical Sciences, Wilberforce Road, Cambridge CB3 0WA}
\email{jw2006@cam.ac.uk}

\definecolor{imperiallight}{RGB}{24,142,179}
\definecolor{grn}{RGB}{68, 179, 24}
\definecolor{purple}{RGB}{129,39,137}
\hypersetup{
    breaklinks=true,
	linkcolor={purple},
	citecolor={purple},
	urlcolor={purple}
}
\usetikzlibrary{ decorations.markings}

\begin{document}
\begin{abstract}
   We investigate \textit{accessible subgroups} of a profinite group $G$, i.e. subgroups $H$ appearing as vertex groups in a graph of profinite groups decomposition of $G$ with finite edge groups. We prove that any accessible subgroup $H \leq G$ arises as the kernel of a continuous derivation of $G$ in a free module over its completed group algebra. This allows us to deduce splittings of an abstract group from splittings of its profinite completion. We prove that any finitely generated subgroup $\Delta$ of a finitely generated virtually free group $\Gamma$ whose closure is a factor in a profinite amalgamated product $\Gh = \Db \amalg_K L$ along a finite $K$ must be a factor in an amalgamated product $\Gamma = \Delta \am \Lambda$ along some $\Chi \cong K$. This extends previous results of Parzanchevski--Puder, Wilton and Garrido--Jaikin-Zapirain on free factors.
\end{abstract}
\maketitle
\section{Introduction}
The structure of profinite groups is as intriguing as it is elusive. As Galois groups, profinite groups appear throughout mathematics, and it is of great interest in fields such as number theory or algebraic geometry to shed light on the peculiar structural properties of these objects. Within group theory, particular focus lies on the study of \emph{profinite rigidity}, that is, the question to what extent the structure of an abstract group $\Gamma$ is visible in its finite quotients. More concisely, this problem can be rephrased as the question of how much the structure of an abstract group $\Gamma$ is visible in its profinite completion $\Gh$, the latter being a profinite group comprised of the inverse system of the finite quotients of $\Gamma$ and their epimorphisms.

A natural question regarding the structure of a mathematical object is whether it can be decomposed via the canonical constructions in the category it inhabits, that is, limits and colimits. While limits commute with the forgetful functor from the profinite to the abstract category, colimits do not: already the coproduct of two profinite groups, their profinite free product, can significantly outgrow their coproduct as abstract groups, the free product. This makes the study of profinite coproduct decompositions both harder and more interesting. A larger class of colimit decompositions is given by \emph{graph of groups splittings}, which are related via Bass--Serre theory to a group's ability to act on trees. In recent years, a beautiful profinite analogue of this geometric theory---profinite groups acting on profinite trees---has come to fruition, albeit in a far more obscure form than its abstract sibling. This development has brought about answers as well as questions.

It is natural then to ask whether profinite decompositions of the profinite completion $\Gh$ of a residually finite group $\Gamma$ must arise from abstract decompositions of $\Gamma$. In other words: can one recognise decompositions of $\Gamma$ in its finite quotients? The restriction to the class of residually finite groups is necessary here to ensure that $\Gamma$ embeds into its profinite completion $\Gh$ as a subgroup. A celebrated instance of this general question, attributed to Remeslennikov (cf. \cite[Question 15]{Remeslennikov}), asks whether there exists a finitely generated residually finite group which is not free but whose profinite completion is free profinite. The solution remains out of sight at the time being. However, in recent years, the related question of whether a subgroup $\Delta \leq \Gamma$ whose closure $\Db$ in $\Gh$ is a factor in a profinite decomposition must itself arise as a factor in an abstract decomposition of $\Gamma$ has gained considerable attention. In the case of free factors of finitely generated free groups, this was answered positively by Parzanchevski and Puder in \cite{Parzanchevski_Puder} and by Wilton in \cite{Wilton_Free_Factors}. Last year, Garrido and Jaikin-Zapirain generalised in \cite{Garrido_Jaikin} this result to free factors of finitely generated virtually free groups, providing a Bass--Serre theoretic proof.

Thus, one is tempted to extend the overarching question to a more general class of decompositions. The natural next step would be to consider free products with amalgamation, which arise for instance as the fundamental groups of unions of path-connected spaces with path-connected intersection. Given a residually finite group $\Gamma$ which splits as a free product $\Gamma = \Delta \am \Lambda$ with amalgamation along a finite $\Chi$, it follows from the definition of pushouts (cf. \cite[Excercise 9.2.7(3)]{RZ}) that $\Gh = \Db \amalg_{\upchi} \overline{\Lambda}$, where $\amalg$ denotes the profinite amalgamated product, and $\Db$ and $\overline{\Lambda}$ denote the closures of $\Delta$ and $\Lambda$ in $\Gh$, respectively. We ask here the converse direction of this implication, which is considerably more involved and vastly unanswered. We remark the particular significance of this question in light of the recent discovery (cf. \cite[Corollary 1.2]{property_fa}) that splitting as an amalgam is not a profinite invariant among finitely presented residually finite groups.
\begin{question}\label{Question}
    Let $\Gamma$ be a finitely generated residually finite group and $\Delta \leq \Gamma$ a finitely generated subgroup. Assume that $\Gh = \Db \amalg_K L$ for some profinite groups $K,L$. Are there abstract groups $\Chi, \Lambda$ such that $\Gamma = \Delta \am \Lambda$ and $\Chi \cong K$?
\end{question}
The theory of Bass and Serre is not unique among those which infer the structure of a group from its ability to act on other mathematical objects. The topologically inspired---yet entirely algebraic---theory of groups acting on modules, group cohomology, has proven instrumental in the study of groups in the last century. In 1979, Dunwoody related these two theories in an elegant result, which states that a subgroup $\Delta$ of a group $\Gamma$ is contained in the kernel of a non-trivial derivation of $\Gamma$ in its group ring $R[\Gamma]$ if and only if $\Delta$ is \emph{accessible within} $\Gamma$, i.e. $\Delta$ is contained in the stabiliser of a vertex under an action of $\Gamma$ on a tree with finite edge stabilisers. We introduce the following analogous definition for profinite groups.
\begin{definition}
Let $G$ be a profinite group. A closed subgroup $H \leq G$ is \emph{accessible within $G$} if $G = \PT$ for some graph of profinite groups $\gxt$ over a finite graph $\Xi$ such that all edge groups $\Gs(e)$ for $e \in E(\Xi)$ are finite and there exists a vertex $v \in \Gs(v)$ with $H = \Gs(v)$.
\end{definition}
When necessary to specify the graph of groups, we shall say that $H$ is \emph{accessible within $G$ via} $\gxt$. We refer the reader to Section \ref{Sec::Prelim:Trees} for details on graphs of profinite groups. The cohomology theory of profinite groups emerges analogously to that of abstract groups, and in the case of a profinite completion $\iota \colon \Gamma \to \Gh$, one may relate the corresponding theories for $\Gamma$ and $\Gh$ in a variety of ways. This was exploited partially by Garrido--Jaikin-Zapirain in \cite{Garrido_Jaikin}, who used Dunwoody's theorem to construct an abstract graph of groups splitting using a continuous derivation obtained from a profinite free product. In this paper, we address the question of obtaining these continuous derivations for general graphs of profinite groups. We prove the following profinite analogue of the ``if'' direction in Dunwoody's result.

\begin{restatable}{theorem}{MTA}
\label{Thm::Subgroup Accessibility}
	Let $G$ be a profinite group and $H \leq G$ a closed subgroup accessible within $G$ via a graph of profinite groups $\gxt$ with $n$ edges. If $R$ is a finite ring whose characteristic does not divide the order of any edge group $\Gs(e)$ for $e \in E(\Xi)$, then there exists a continuous derivation
	\[
	f \colon G \to R[[G]]^n
	\]
	satisfying $\Ker(f) = H$.
\end{restatable}
The difficulty in erecting such a derivation in the profinite setting is that, first, the minimal subtree of a profinite tree containing two vertices may be infinite, so Dunwoody's construction using almost invariant sets does not carry over, and second, an arbitrary element of a graph of profinite groups does not have a normal form, so precise control over the derivation's kernel becomes a question of asymptotics within a profinite group. We settle the first issue in the proof of Proposition~\ref{Prop::Partial Subgroup Accessibility}, and dedicate the proof of Theorem~\ref{Thm::Subgroup Accessibility} to the resolution of the second issue.
 
Having established Theorem~\ref{Thm::Subgroup Accessibility}, we aim to utilise it together with the abstract opposite direction to deduce splittings of an abstract group from splittings of its profinite completion. In \cite{Garrido_Jaikin}, Garrido and Jaikin-Zapirain follow this strategy in the case of free products. They start with a finitely generated virtually free group $\Gamma$ and a finitely generated subgroup $\Delta$ whose closure in $\Gh$ is a factor in a profinite free splitting $\Gh = \Db \amalg L$. Using a direct sum decomposition of augmentation ideals inherent to free products, they obtain a continuous derivation $f$ of $\Gh$ in its completed group algebra $R[[\Gh]]$ for a finite ring $R$. Restricting $f$ to $\Gamma$, they prove using homological algebra that the image $f(\Gamma)$ lies in a projective $R[\Gamma]$-module. Dunwoody's theorem then yields a graph of groups decomposition of $\Gamma$, which they mould to the form $\Gamma = \Delta \am[] \Lambda$ for some $\Lambda \leq \Gamma$.

In this paper, we employ a similar idea to attack Question \ref{Question}. The difficulty in generalising to free products with non-trivial amalgamation is twofold. First, there no longer exists a direct sum decomposition of augmentation ideals, so the continuous derivation of $\Gh$ must be obtained by alternative means. We do this in Theorem~\ref{Thm::Subgroup Accessibility} for all graphs of groups. The second difficulty lies in the increased complexity of Bass--Serre theoretic manipulations necessary to retrieve the desired decomposition. This we accomplish in Section \ref{Sec::Amalgams}. Thus, we prove the following result, which answers Question \ref{Question} positively in the class of virtually free groups\footnote{This question was also asked by Ian Agol at the \textit{Workshop on Profinite Rigidity} held at the ICMAT in Madrid, June 2023.}.

\begin{restatable}{theorem}{MTB}
	\label{Thm::Recognition}
	Let $\Gamma$ be a finitely generated virtually free group, let $\Delta \leq \Gamma$ be a finitely generated subgroup and let $K \leq \Gh$ be a finite subgroup of the profinite completion. The following are equivalent:
	\begin{enumerate}[(a)]
		\item there exists a closed subgroup $L \leq \Gh$ such that $\Gh = \Db \amalg_K L$;
		\item there exist subgroups $\Lambda, \Chi \leq \Gamma$ and $h \in \Db$ with $\Gamma = \Delta \am \Lambda$ and $\Chi = h^{-1}Kh$.
	\end{enumerate}
\end{restatable} 

The article is structured as follows. In Section \ref{Sec::Prelim}, we review definitions and preliminary results on the topics of abstract and profinite Bass--Serre theory and profinite group cohomology, which we shall require in our proofs.
In Section \ref{Sec::Access}, we construct certain continuous derivations of profinite groups which split as graphs of profinite groups. We prove a general result on the structure of continuous derivations of graphs of profinite groups in Proposition~\ref{Prop::Glue}, and utilise this result to demonstrate Proposition~\ref{Prop::Partial Subgroup Accessibility}, the profinite analogue of Dunwoody's result on the construction of derivations. Sharpening explicit control over the kernel of these derivations, we obtain Theorem~\ref{Thm::Subgroup Accessibility}.
Finally, Section \ref{Sec::Amalgams} is devoted to the proof of Theorem~\ref{Thm::Recognition}. The major effort lies in the Bass--Serre theoretic manipulations necessary to restore the required amalgamated splitting, which we split into a few lemmata. The substance is contained in the proofs of Lemma~\ref{Lem::Bass--Serre Magic 1} and Theorem~\ref{Thm::Recognition}.

\section*{Acknowledgements}
The author is grateful to his PhD supervisor, Gareth Wilkes, for suggesting this question and for invaluable discussions throughout the realisation of the project. Thanks are also due to Henry Wilton for an inspiring conversation on profinite acylindricity. Financially, the author was supported in his research by a Cambridge International Trust and King's College Scholarship. Part of this work was done while the author was visiting the Instituto de Ciencias Matem\'aticas, Madrid, for whose hospitality he is thankful. Finally, the author is grateful to the anonymous referee for helpful suggestions and for pointing out references which allowed for the shortening of some proofs in Section~\ref{Sec::Amalgams}.

\section{Preliminaries}\label{Sec::Prelim}
In this section, we outline preliminary results on the topics of profinite groups acting on profinite trees and the homological algebra of profinite groups. In the interest of brevity, we outline only the key results used in our proofs in later sections. For more detailed expositions of these topics, we invite the reader to \cite{Zalesskii_Melnikov}, \cite{Ribes_Graphs}, \cite{Pro_P_Bass_Serre} and \cite{Wilkes_PhD}, as well as \cite{RZ}, \cite{Serre_Cohomology} and \cite{Gille_Szamuely}, respectively.

\subsection{Profinite Groups Acting on Profinite Trees}
\label{Sec::Prelim:Trees}
A \textit{profinite graph} $X$ is a compact, Hausdorff and totally disconnected space equipped with a distinguished closed subspace $V(X) \subseteq X$, called the \emph{vertex set}, as well as two continuous maps $d_0, d_1 \colon X \to V(X)$ satisfying $d_0\at{V(X)} = d_1\at{V(X)} = \operatorname{id}_{V(X)}$. A \emph{quasi-morphism} of graphs $f \colon X \to Y$ is a continuous map which commutes with $d_0$ and $d_1$. We say that $X$ is \emph{connected} if every finite image of $X$ under a quasi-morphism is connected as an abstract graph.
By \cite[Proposition 2.1.10]{Ribes_Graphs}, any profinite graph $X$ has a unique decomposition into its connected components, that is, maximal connected profinite subgraphs. Recall further that a connected profinite graph $\Gg$ is a \emph{tree} if the sequence
\begin{equation}
    0 \to \pZ[[(E^*(\Gg),\ast)]] \xrightarrow{d} \pZ[[V(\Gg)]] \xrightarrow{\varepsilon} \pZ \to 0
\end{equation} is exact; here, $\pZ[[Y]]$ denotes the profinite free module on a (pointed) space $Y$, $E^*(\Gg) = (X / V(X),\ast)$ is the pointed quotient space with $\ast = [V(X)]$, and $d, \varepsilon$ are the morphisms of profinite modules generated by $d(e) = d_1(e) - d_0(e)$, $d(\ast) = 0$ and $\varepsilon(v) = 1$ for $e \in E(\Gg)$ and $v \in V(\Gg)$. By \cite[Proposition 2.4.9]{Ribes_Graphs}, the intersection of an arbitrary collection of subtrees of a tree $\Gg$ is either empty or a tree. It follows that for any two vertices $v,w \in V(\Gg)$, there exists a minimal subtree containing $v$ and $w$, which we shall denote as $[v,w] \subseteq \Gg$.

A profinite group $G$ is said to \emph{act} on a profinite graph $\Gg$ if there is a continuous group action $G \curvearrowright \Gg$ which commutes with the adjacency maps $d_0,d_1$. If $U$ is an open normal subgroup of $G$, we shall write $\pi_U \colon G \to \GU$ for the canonical projection of groups and $p_U \colon \Gg \to U \bs \Gg$ for the associated quotient of profinite graphs. In that case, the finite group $\GU$ acts naturally on the profinite graph $U \bs \Gg$ via $\pi_U(g) \cdot p_U(x) = p_U(g \cdot x)$. There exist appropriate definitions of graphs of profinite groups and their profinite fundamental groups over any profinite graph---this is done using sheaves of profinite groups, see for instance \cite[Sections 5-6]{Ribes_Graphs}. However, for the purposes of the present work, it shall be enough to carry out these definitions in the case where the underlying graph $X = \Xi$ is finite. A \emph{graph of profinite groups} $\ggx$ over a finite connected graph $\Xi$ consists of:
\begin{enumerate}[(a)]
    \item a collection of profinite groups $\{\Gs(x) : x \in \Xi\}$, referred to as the \emph{vertex} and \emph{edge groups}; and
    \item two collections of continuous injective homomorphisms of profinite groups $\{\partial_0 \colon \Gs(e) \to \Gs(d_0(e))\}_{e \in E(\Xi)}$ and $\{\partial_1 \colon \Gs(e) \to \Gs(d_1(e))\}_{e \in E(\Xi)}$, referred to as the \textit{edge inclusions}.
\end{enumerate}
Given a graph of profinite groups $\ggx$, we want to associate to it a fundamental group, which will be a group generated by the vertex groups and a collection of stable letters $T = \{t_e : e \in E(\Xi)\}$, indexed by the edges of $\Xi$, which are glued together with respect to the structure of $\Xi$. Specifically, choose a spanning tree $\Theta$ of $\Xi$ and consider the profinite group given by the profinite free product
\[
W \ggx =  \coprod_{v \in V(\Xi)} \Gs(\tau) \amalg \mathscr{F}(T)
\]
where $\mathscr{F}(T)$ denotes the free profinite group on the space $T$. Let $N \gxt \trianglelefteq W \ggx$ be the minimal closed normal subgroup containing the sets $\{t_e : e \in E(\Theta)\}$ and $\{\partial_1(g)^{-1}t_e^{-1}\partial_0(g)t_e : e \in E(\Xi), g \in \Gs(e)\}$. The \emph{profinite fundamental group of $\ggx$ with respect to $\Theta$} is defined as
\[
\PP = \frac{W\ggx}{N\gxt}
\]
which is profinite as $N\ggx$ is closed in $W\ggx$. Given a different choice of spanning tree $\Theta'$ of $\Xi$, there is an isomorphism
\begin{equation}
	\PT \cong \PT[(\Gs, \Xi, \Theta')]
\end{equation}
although this isomorphism may carry vertex groups to distinct conjugates. Thus, we shall write $\PP$ whenever we refer to the isomorphism type of this group, and specify a spanning tree only when we need to refer to the images of vertex groups in the fundamental group. Moreover, to distinguish this construction from the abstract fundamental group of a graph of groups $\ggx$, we shall write $\PA$ to denote the latter. As an example, consider the following graph of profinite groups $\ggx$ given by
\begin{center}
	\begin{tikzpicture}[decoration={markings, 
			mark= at position 0.53 with {\arrow{stealth}}}]
		\draw[fill=black] (-1,0) circle (1pt);
		\draw[fill=black] (1,0) circle (1pt);
		\node at (-1,-0.3) {$H$};
		\node at (1,-0.3) {$L$};
		\node at (0,0.3) {$K$};
		\draw [postaction={decorate}] (-1,0) -- (1,0);
	\end{tikzpicture}
\end{center}
where $H$ and $L$ are profinite groups and $K$ is a common subgroup. Write $G = \PP$ for its profinite fundamental group. Then $G = H \amalg_K L$ is the \emph{profinite amalgamated product} or \emph{profinite free product with amalgamation} along $K$, i.e. the pushout
\begin{equation}
\begin{tikzcd}
	K \arrow[d] \arrow[r] & H \arrow[d] \arrow[ddr, dotted, "\forall f_H", bend left = 30]& \\
	L  \arrow[drr, dotted, "\forall f_L", bend right = 30] \arrow[r] & G \arrow[dr, dotted, "\exists f"]& \\
	&& \forall A
\end{tikzcd}
\end{equation}
in the category of profinite groups. Unlike in the abstract setting, the canonical morphisms $\nu_x \colon \Gs(x) \to \PT$ for $x \in \Xi$ may not be injective: see \cite[Example 9.2.10]{RZ} for a counterexample. We say that $\gxt$ is an \emph{injective} graph of profinite groups if it is indeed the case that the canonical morphisms $\nu_x$ are injective for all $x \in \Xi$. We note that a graph of profinite groups which is not injective can be replaced with a natural construction which does form an injective graph of profinite groups: see \cite[Section 6.4]{Ribes_Graphs}. Moreover, as we shall see in Proposition \ref{Prop::profinite graph of profinite groups}, a graph of profinite groups over a finite graph with finite edge groups is automatically injective.

It is an exercise in the definition of pushouts (cf. \cite[Exercise 9.2.7(2)]{RZ}) to show that $\widehat{\Delta \am[K] \Lambda} \cong \widehat{\Delta} \amalg_K \widehat{\Lambda}$ whenever $\Delta, \Lambda$ are residually finite and $K$ is finite. More generally, it is sensible to ask whether the profinite completion functor commutes with the (abstract and profinite) fundamental group construction. The following proposition yields a precise topological description of the relation between the two constructions; it is obtained essentially from \cite[Proposition 6.5.1]{Ribes_Graphs}. The symbols $\trianglelefteq_f$ and $\trianglelefteq_o$ denote finite-index normal subgroups and open normal subgroups in a topological group, respectively.
\begin{proposition}\label{Prop::profinite graph of profinite groups}
    Let $\gxt$ be a graph of profinite groups over a finite graph $\Xi$. The profinite fundamental group $\PT$ is isomorphic to the completion of the abstract fundamental group $\PAT$ with respect to the profinite topology determined by the neighbourhood basis \[
    \mathcal{U} = \{N \trianglelefteq_f \PAT \mid \forall x\in \Xi, \, N \cap \Gs(x) \trianglelefteq_o \Gs(x)\}
    \] for the identity. If, additionally, all edge groups $\Gs(e)$ for $e \in E(\Xi)$ are finite, then $\ggx$ is an injective graph of profinite groups, and the completion map
    \[
        \iota \colon \PAT \to \PT
    \]
    is a monomorphism.
\end{proposition}
\begin{proof}
    The isomorphism between the profinite group $\PT$ and the postulated completion of the abstract group $\PAT$ is established in \cite[Proposition 6.5.1]{Ribes_Graphs}. In \cite[Corollary 1.3]{Ribes71}, it is proven that a profinite amalgamated product of two profinite groups along a finite amalgamation is injective; an analogous proof shows that any graph of profinite groups over a finite graph and with finite edge groups is an injective graph of profinite groups. The injectivity of the completion map $\iota$ is equivalent to injectivity of the graph of groups: for amalgams, this is \cite[Theorem 9.2.4]{RZ}, and the general case is established analogously.
\end{proof}
Note that the profinite topology determined by $\mathcal{U}$ will agree with the full profinite topology on $\Gamma$ if the vertex groups are \emph{strongly complete} profinite groups, meaning that they are isomorphic to their own profinite completions. This is an extremely general condition: it is a deep theorem of Nikolov and Segal (cf. \cite[Theorem 1.1]{Nikolov_Segal}) that every (topologically) finitely generated profinite group is strongly complete. In that case, Proposition~\ref{Prop::completion of fundamental group} yields an isomorphism $\PP = \widehat{\PA}$ of profinite groups. Given an abstract graph of groups $\ggx$ over a finite graph $\Xi$ with fundamental group $\Gamma = \PA$, we define its \emph{completion} $(\overline{\Gs},X)$ to be the graph of profinite groups given by the completion of the groups in $\ggx$ with respect to the profinite topology induced by $\Gamma$. In that case, we may use Proposition~\ref{Prop::profinite graph of profinite groups} to obtain the following.
\begin{proposition}[Proposition 6.5.3 in \cite{Ribes_Graphs}]\label{Prop::completion of fundamental group}
	Let $\ggx$ be a graph of groups over a finite graph $\Xi$ whose fundamental group $\PA$ is residually finite. The completion $(\overline{\Gs},\Xi)$ is an injective graph of profinite groups and there is an isomorphism of profinite groups
	\begin{equation}
		\PP[(\overline{\Gs},X)] \cong \widehat{\PA}
	\end{equation}
	induced by the inclusions.
\end{proposition}
Note that, in general, the profinite topology of a residually finite group $\Gamma$ may induce a strictly coarser topology on a subgroup $\Delta$ than the profinite topology of $\Delta$ itself. Thus, care needs to be taken to distinguish between the profinite groups $\Db$, the closure of $\Delta$ within $\Gh$, and $\widehat{\Delta}$, the profinite completion of $\Delta$, which are not isomorphic in such an event. However, in specific cases, we know that the profinite topology of $\Gamma$ must indeed induce on $\Delta$ its full profinite topology, which means that there is an isomorphism $h \colon \widehat{\Delta} \xrightarrow{\sim} \Db \subseteq \Gh$ making the diagram
\begin{equation}
\begin{tikzcd}
\widehat{\Delta} \arrow[rr, "h"] && \Db \\
&\Delta \arrow[ul, "\iota_{\Delta}"] \arrow[ur, "\iota_{\Gamma}\at{\Delta}"']&
\end{tikzcd}
\end{equation} commute. For instance, it is well-known that this occurs if $\Delta$ has finite index in $\Gamma$; see e.g., \cite[Lemma 3.1.4]{RZ}. Another case is when $\Gamma$ is the fundamental group of a graph of residually finite groups over a finite graph with finite edge groups and $\Delta$ is a vertex group. In that case, we obtain the following corollary to Proposition~\ref{Prop::completion of fundamental group}.

\begin{corollary}[Proposition 6.5.6 in \cite{Ribes_Graphs}]\label{Cor::profinite graph of finite groups}
    Let $\ggx$ be a graph of finite groups over a finite graph $\Xi$. Then $\ggx$ forms an injective graph of profinite groups and
	\begin{equation}
	\PP \cong \widehat{\PA}
	\end{equation}
	is an isomorphism of profinite groups.
\end{corollary}

Corollary \ref{Cor::profinite graph of finite groups} is particularly useful for our purposes, in light of the following celebrated result of classical Bass--Serre theory, which was proven by Karrass, Pietrowski and Solitar in \cite{karrass_pietrowski_solitar}.

\begin{theorem}[Theorem 1 in \cite{karrass_pietrowski_solitar}]\label{Thm::KPS}
	A finitely generated group $\Gamma$ is virtually free if and only if there is a graph of groups $\ggx$ over a finite graph $\Xi$ such that $\Gamma = \PA$ and $\Gs(x)$ is finite for each $x \in \Xi$.
\end{theorem}

Another situation in which the induced topology on a finitely subgroup $\Delta \leq \Gamma$ agrees with the full profinite topology is when the ambient group $\Gamma$ is \emph{subgroup separable}, meaning that finitely generated subgroups are closed in the profinite topology on $\Gamma$. For virtually free groups, this property follows essentially from a result of M. Hall (cf. \cite[Proposition I.3.10]{lyndon_schupp}) and we refer the reader to \cite[Theorem 11.2.2 and Corollary 11.2.6]{Ribes_Graphs} for a geometric proof.
\begin{lemma}[Theorem 11.2.2 in \cite{Ribes_Graphs}]\label{Lem::Induced Topology}
Virtualy free groups are subgroup separable and induce the full profinite topology on their finitely generated subgroups.
\end{lemma}
One would hope for a profinite analogue of the Bass--Serre structure theorem, which postulates a correspondence between the actions of a group $\Gamma$ on trees and graphs of groups with fundamental group $\Gamma$. One direction associates to a group action on a tree $\Gamma \curvearrowright \Theta$ a \emph{quotient graph of groups} $\ggx$, which is constructed by choosing a spanning tree in $\Sigma \subseteq \Xi$ and a lift $\widetilde{\Sigma} \subseteq \Theta$, and then defining $\ggx$ as the stabilisers of $\widetilde{\Sigma}$ under the action of $\Gamma$. Unfortunately, in the profinite setting, this construction does not always work: a profinite graph might not have a spanning subtree (cf. \cite[Example 3.4.1]{Ribes_Graphs}), and even if it does, that tree might not have lift to the covering tree (cf. \cite[Example 3.4.2]{Ribes_Graphs}). However, in the case where the quotient $X = G \bs T$ is finite, it forms a bona fide graph of groups, so it is possible to construct a quotient graph of groups, and one obtains a profinite analogue of the Bass--Serre structure theorem. We refer the curious reader to Section 6 in \cite{Ribes_Graphs} for details and to \cite[Theorem 6.6.1]{Ribes_Graphs} for the statement of the structure theorem, which we omit here in the interest of brevity. Recently, a general structure theorem has also been established in the pro-$p$ category: see \cite[Theorem 1.1]{Pro_P_Bass_Serre}.

The other direction associates to a graph of profinite groups $(\Gs,X)$ a profinite tree $T$ referred to as the \emph{profinite structure tree} whereon the fundamental group $G = \PP[(\Gs,X)]$ acts naturally and whose quotient by this action retrieves the initial graph. Although this construction can be carried out for any graph of profinite groups (cf. \cite[Section 6.3]{Ribes_Graphs}), we shall only presently describe it in the case where $X$ is a finite graph, as we will not require the complexity of greater generality. In that case, choose a spanning subtree $\Theta \subseteq X$ and write $G(x) \leq G$ for the image of $\Gs(x)$ in $G$ under the inclusion specified by $\Theta$. Define the profinite tree $T$ as
\[
V(T) = \bigsqcup_{v \in V(X)} G / G(v) \quad \text{and} \quad E(T) = \bigsqcup_{e \in E(X)} G / G(e)
\]
where $G/G(x)$ denotes the space of left cosets of $G(x) \leq G$ endowed with the (profinite) quotient topology, and
\[
d_0(gG(e)) = gG(d_0(e)) \quad \text{and} \quad d_1(gG(e)) = gt_eG(d_1(e))
\] whenever $e \in E(X)$ is an edge and $t_e$ the associated stable letter. One verifies that $T$ forms a profinite graph, and in fact a profinite tree: see \cite[Corollary 6.3.6]{Ribes_Graphs}. Moreover, as each $G/G(e)$ is closed in $T$ and $X$ is finite, we obtain:
\begin{remark}
    \label{Rem::Closed Edge Set} The structure tree of a graph of profinite groups over a finite graph has closed edge set.
\end{remark}

The profinite structure tree construction is especially useful in that it allows one to study profinite groups which decompose as graphs of profinite groups via their actions on the structure tree. This was first exploited by Zalesskii and Melnikov in \cite{Zalesskii_Melnikov}, where certain structural properties of closed subgroups of such profinite groups were deduced from their action on the profinite structure tree. For instance, they prove the following result, which we shall utilise throughout the present work.

\begin{proposition}[Theorem 2.10 in \cite{Zalesskii_Melnikov}]\label{Prop::Finite Group Profinite Tree}
    A finite group acting on a profinite tree must fix a vertex. 
\end{proposition}

As a consequence, one finds a variety of structural results regarding subgroups of such profinite groups: see \cite[Section 3]{Zalesskii_Melnikov}. We refer the curious reader also to \cite{Bridson_Conder_Reid} for a recent application of this result in the proof of an elegant theorem on the profinite rigidity of Fuchsian groups.

\subsection{Homological Algebra of Profinite Groups}\label{Sec::Prelim:Coh}
Let $G$ be a profinite group and $R$ a profinite ring. A left/right \emph{$G$-module} is an abelian topological group $M$ equipped with a left/right continuous $G$-action. The completed group algebra of $G$ over $R$ is the profinite ring
\[
R[[G]] = \lim_{\longleftarrow U} \left(R \left[G/U \right] \right)
\]
where $U \trianglelefteq_f G$ runs over all open normal subgroups. The abstract group algebra $R[G]$ may be endowed with the natural topology generated by the cosets of the ideals which arise as kernels of the quotient epimorphisms 
\[
R[G] \longrightarrow (R/I)[G/U]
\] for $I \trianglelefteq_o R$ an open ideal and $U \trianglelefteq_o G$ an open subgroup. By \cite[Lemma 5.3.5]{RZ}, $R[G]$ embeds in $R[[G]]$ and the latter is its completion with respect to the above topology. Given a $G$-module and a continuous action $R \curvearrowright M$, one obtains a canonical $R[[G]]$-module structure on $M$, where the action of $G$ extends $R$-linearly to the image of $R[G]$ as a dense subset of $R[[G]]$, and then extends continuously to all of $R[[G]]$. We define the cohomology groups of a $G$-module $M$ via
\[
H^\bullet(G,M) = \lim_{U \longrightarrow} H^\bullet(G/U,M^U)
\]
where $U \trianglelefteq_o G$ runs over all open normal subgroups of $G$. Note that given an abstract group $\Gamma$ with profinite completion map $\iota \colon \Gamma \to \Gh$, any $\Gh$-module $M$ forms a $\Gamma$-module with action via $\iota$. Conversely, if $M$ is a \emph{finite} $\Gamma$-module, the action $\Gamma \to \operatorname{Aut}(M)$ factors through a map $\Gh \to \operatorname{Aut}(M)$, making $M$ a discrete $\Gh$-module. We say that $\Gamma$ is \emph{cohomologically good} if the profinite completion map $\iota \colon \Gamma \to \Gh$ induces an isomorphism of cohomology groups
\[
\iota^* \colon H^\bullet(\Gh,M) \longrightarrow H^\bullet(\Gamma,M)
\]
for any finite $\Gamma$-module $M$. The following result yields an equivalent characterisation of cohomological goodness for $\mathrm{FP}_\infty$-groups. It was proven first by Wilkes in \cite[Proposition 3.14]{Wilkes_SFS} and subsequently, in the below formulation, by Jaikin-Zapirain in \cite[Proposition 3.1]{jaikin_fiberedness}.

\begin{lemma}[Proposition 3.1 in \cite{jaikin_fiberedness}]\label{Lem::Resolution of pZ}
	Let $\Gamma$ be a group of type $\mathrm{FP}_\infty$ and
	\begin{equation}
		\cdots \rightarrow \Z[\Gamma]^{n_i} \xrightarrow{\phi_i} \Z[\Gamma]^{n_i-1} \rightarrow \ldots \rightarrow \Z[\Gamma]^{n_1} \xrightarrow{\phi_1} \Z[\Gamma] \rightarrow \Z \rightarrow 0 
	\end{equation}
	be a resolution of the trivial $\Gamma$-module $\Z$ by finitely generated free $\Z[\Gamma]$-modules. Then $\Gamma$ is cohomologically good iff the induced sequence
	\begin{equation}
		\cdots \rightarrow \pZ[[\Gh]]^{n_i} \xrightarrow{\widehat{\phi_i}} \pZ[[\Gh]]^{n_i-1} \rightarrow \ldots \rightarrow \pZ[[\Gh]]^{n_1} \xrightarrow{\widehat{\phi_1}} \pZ[[\Gh]] \rightarrow \pZ \rightarrow 0 
	\end{equation}
	is exact.
\end{lemma}

We conclude this subsection with a result of Garrido and Jaikin-Zapirain on the homological algebra of the completed group algebra of the profinite completion $\Gh$ of a virtually free group $\Gamma$, as an abstract $\Gamma$-module. We shall use it as a portal from the profinite to the abstract world, which allows us to pass derivations of $\Gh$ in its completed group algebra to derivations of $\Gamma$ in a meaningful $\Gamma$-module.

\begin{theorem}[Corollary 5.6 in \cite{Garrido_Jaikin}]\label{Thm::Homological Algebra}
	Let $\Gamma$ be a finitely generated virtually free group. For any prime $p$, the left $\Fp[\Gamma]$-module $\Fp[[\Gh]]$ is isomorphic to a direct union of free left $\Fp[\Gamma]$-modules.
\end{theorem}

We note that $R[[\Gh]]$ continues to be isomorphic to a direct limit of projective $R[\Gamma]$-modules for certain other finite rings, such as $\Z/n\Z$ for $n$ a product of distinct primes: cf. \cite[Corollary 5.7]{Garrido_Jaikin}. However, this shall not be necessary for our purposes, as we substitute their argument on supernatural numbers by a Bass--Serre theoretic argument (cf. Lemma~\ref{Lem::Bass--Serre Magic 1}).

\subsection{A juxtaposition}\label{Sec::Prelim:Juxta} Pursuant of a Bass--Serre theoretic description of groups of cohomological dimension one, Dunwoody proved in 1979 a remarkable result juxtaposing the worlds of Bass--Serre theory and group cohomology. Recall that a \emph{left derivation} or \emph{crossed homomorphism} of a group $\Gamma$ in a left $\Gamma$-module $M$ is a map $f \colon \Gamma \to M$ satisfying $f(gh) = f(g) + g f(h)$ for all $g,h \in \Gamma$. Similarly, a \emph{right derivation} in a right $\Gamma$-module $M$ is a map $f \colon \Gamma \to M$ satisfying $f(gh) = f(g)h + f(h)$ for all $g,h \in \Gamma$. Note that right derivations of $\Gamma$ in a right $\Gamma$-module $M$ are equivalently left derivations of the (isomorphic) opposite group $\Gamma^{\mathrm{op}}$ in $M$ endowed with its natural left $\Gamma^{\mathrm{op}}$-module structure. For this reason, it is inconsequential whether one works with left or right derivations; we choose right derivations out of convenience and agreement with the convention of Dicks and Dunwoody. Hence, unless stated otherwise, all derivations appearing in this article shall be assumed \emph{right} derivations unless explicitly mentioned otherwise.

Given a right $\Gamma$-module $M$, one may form the right semi-direct product $\Gamma \ltimes M$, which is a group with underlying set $\Gamma \times M$ and group operation $(g_1,m_1) \cdot (g_2,m_2) = (g_1g_2, m_1g_2 + m_2)$. It comes equipped with natural projections of sets
\[
    \operatorname{proj_1} \colon \Gamma \ltimes M \to \Gamma \quad \text{and} \quad \operatorname{proj_2} \colon \Gamma \ltimes M \to M
\]
where we note that $\operatorname{proj_1}$ forms a homomorphism of groups but $\operatorname{proj_2}$ might not. A \emph{homomorphic section} of $\operatorname{proj_1}$ is then a homomorphism of groups $F \colon \Gamma \to \Gamma \ltimes M$ which satisfies $\operatorname{proj}_1 F = \operatorname{id}_\Gamma$. We obtain the following alternative description of right derivations, which shall be useful in the proofs of Section~\ref{Sec::Access}.
\begin{remark}\label{Rem::Sections}
    Let $\Gamma$ be a group and $M$ a right $\Gamma$-module. There exists a bijective correspondence between right derivations $f \colon \Gamma \to M$ and homomorphic sections $F \colon \Gamma \to \Gamma \ltimes M$ of the projection $\operatorname{proj}_1 \colon \Gamma \ltimes M \to \Gamma$, given by $f \mapsto (\operatorname{id}_\Gamma, f)$ on derivations and $F \mapsto \operatorname{proj}_2 F$ on sections.
\end{remark}
The analogous statement for left derivations is explained in \cite[Section IV.2]{brown_cohomology}; we invite the reader to verify that the present form applies to right derivations. Naturally, both the equivalence of left and right derivations, as well as their equivalence to homomorphic sections of the semi-direct product $\Gamma \ltimes M$, continue to hold in the profinite category (see \cite[Section 6.8]{RZ}, where derivations and homomorphisms are continuous and $\Gamma \ltimes M$ carries the product topology. We shall make ample use of this fact in Section~\ref{Sec::Access}.

In the cohomology of abstract and profinite groups, derivations and continuous derivations appear as 1-coboundaries. The first cohomology group of an abstract or profinite group is thus given by the set of its (continuous) derivations modulo its (continuous) \emph{inner} derivations, that is, functions $\Gamma \to M$ of the form $g \mapsto m(g - 1)$ for some $m \in M$. We refer the reader to \cite[Section IV.2]{brown_cohomology} and \cite[Section 6.8]{RZ} for details on the abstract and profinite versions of the theory, respectively. It was precisely this characterisation which Dunwoody used to interlink Bass--Serre theory and group cohomology.

\begin{theorem}[Theorem 5.2 in \cite{dunwoody}]\label{Thm::Dunwoody}
	Let $\Gamma$ be a group and $\Delta \leq \Gamma$ a subgroup such that $\Gamma$ is generated by $\Delta$ together with a finite set of elements. For any ring $R \neq 0$, the following are equivalent:
	\begin{enumerate}[(a)]
		\item there exists a graph of groups $\gxt$ over a finite graph $\Xi$ with finite edge groups satisfying $\Gamma = \PAT$ and $\Delta \leq \Gs(v)$ for some $v \in V(\Xi)$;
		\item there exists a non-zero derivation $f \colon \Gamma \to R[\Gamma]$ such that $\Delta \leq \Ker(f)$.
	\end{enumerate}
\end{theorem}

We shall construct a profinite analogue of the direction $(a) \Rightarrow (b)$ in Proposition~\ref{Prop::Partial Subgroup Accessibility}, which we then refine to yield Theorem~\ref{Thm::Subgroup Accessibility}. Together with the abstract direction $(b) \Rightarrow (a)$ in Theorem~\ref{Thm::Dunwoody}, this will enable us to prove in Theorem~\ref{Thm::Recognition} that certain amalgamated factors can be recognised in the profinite completions of virtually free groups. Specifically, we shall use the following refinement of the abstract direction $(b) \Rightarrow (a)$ which was proven by Dicks in \cite{dicks}.

\begin{theorem}[Theorem III.4.6 in \cite{dicks}]\label{Thm::Dicks_Dunwoody}
	Let $\Gamma$ be a group, $R \neq 0$ a ring, $P$ a projective $R[\Gamma]$-module and $f \colon \Gamma \to P$ a derivation. If $\Gamma$ is generated by $\Ker(f)$ together with a finite set of elements, then there exists a graph of groups $\ggx$ with finite edge groups such that $\Gamma = \PA$ and $\Ker(f) = \Gs(v)$ for some vertex $v \in V(\Xi)$.
\end{theorem}

\section{Profinite Subgroup Accessibility and Continuous Derivations} \label{Sec::Access} In this section, we prove the first of our two main results, Theorem~\ref{Thm::Subgroup Accessibility}, which associates to any accessible subgroup $H$ of a profinite group $G$ a continuous derivation of $G$ in a free module over its completed group algebra with kernel $H$. We commence with a general statement about the structure of continuous derivations on a graph of profinite groups over a finite graph.

\begin{proposition}\label{Prop::Glue}
	Let $\gxt$ be an injective graph of profinite groups over a finite graph $\Xi$, write $G = \PT$ and $T = \{t_e : e \in E(\Xi)\}$, and let $M$ be a profinite right $G$-module. Given a collection of continuous derivations $\{f_v \colon G(v) \to M \}_{v \in V(\Xi)}$ and an assignment $f_T \colon T  \to M$ satisfying \begin{equation}
		\left(f_{d_1(e)}(\partial_1(g)) - f_T(t_e)\right)t_e^{-1} \partial_0(g)^{-1}t_e +  f_{d_0(g)}(\partial_0(g)^{-1})t_e + f_T(t_e) = 0 \label{Eq::Lem_Ass}
	\end{equation}
	whenever $e \in E(\Xi)$ and $g \in \Gs(e)$, there exists a unique continuous derivation $ f \colon G \to M $ which extends $f_T$ on $T$ and satisfies $f\at{\Gs(v)} = f_v$ for any $v \in V(\Xi)$.
\end{proposition}
\begin{proof}
	Having fixed the spanning tree $\Theta$ for $\Xi$, we identify each group $\Gs(x)$ with its image $\Gs(x) \leq G = \PT$. By definition of the profinite fundamental group (cf. Section \ref{Sec::Prelim:Trees}), we have $G = W/N$, where
	\[
	W = \coprod_{v \in V(\Xi)} \Gs(v) \amalg \mathscr{F}(T)
	\]
	and $N$ is the minimal closed normal subgroup containing the sets $\{t_e : e \in E(\Theta)\}$ and $\{\partial_1(g)t_e^{-1}\partial_0(g)^{-1}t_e : e \in E(\Xi), g \in \Gs(e)\}$. Given a vertex $v \in V(\Xi)$, define the mapping $F_v \colon \Gs(v) \to G \ltimes M$ as $F(g) = (g,f_v(g))$, which forms a continuous homomorphism into the right semi-direct product $G \ltimes M$ (cf. Remark~\ref{Rem::Sections}). On the other hand, given an edge $e \in E(\Xi)$, define the continuous homomorphism $F_e \colon \sge \to G \ltimes M$ as follows. If $e \in E(\Theta)$, then $t_e = 1$ and we set $F_e = 0$. Otherwise, assume that $e \notin E(\Theta)$, and observe $\langle t_e \rangle \cong \Z$ and $\sge \cong \pZ$ as $W$ has been constructed to include the full profinite completion $\pZ \cdot t_e$ of $\langle t_e \rangle$. Now $\langle t_e \rangle \cong \Z$ is cohomologically good and the trivial $\langle t_e \rangle$-module $\Z$ has free resolution given by
	\begin{equation}
		0 \rightarrow \Z[\langle t_e \rangle] \xrightarrow{\mu} \Z[\langle t_e \rangle] \xrightarrow{\varepsilon} \Z \rightarrow 0 
	\end{equation}
	where $\mu$ is multiplication by $(t_e - 1)$. Thus, Lemma~\ref{Lem::Resolution of pZ} gives a short exact sequence of abelian profinite groups
	\begin{equation}
	0 \rightarrow \pZ[[\sge]] \xrightarrow{\mu} \pZ[[\sge]] \xrightarrow{\varepsilon} \pZ \rightarrow 0 
	\end{equation}
	where $\mu(x) = x(t_e-1)$ and which is split as $\pZ$ is free profinite. Hence, there exists a retraction $\rho$ of $\mu$, and we set
	\[
	F_e(g) = (g,\rho(g-1)\cdot f_T(t_e))
	\]
	where the action of $\pZ[[\sge]]$ on $M$ is induced by the action of $\pZ[[G]]$. Having defined $F_v$ and $F_e$, the universal property of the coproduct yields a homomorphism $F \colon W \to G \ltimes M$ which extends $F_v$ and $F_e$ for all $v \in V(\Xi)$ and $e \in E(\Xi)$. We claim that $F\at{N} = 0$; for this it will suffice to show that $F$ is zero on all generators of $N$ as a closed normal subgroup. If $e \in E(\Theta)$ then $F(t_e) = 0$ by definition. On the other hand, if $e \in E(\Xi) - E(\Theta)$ and $g \in \Gs(e)$, then $F(\partial_1(g)t_e^{-1}\partial_0(g)^{-1}t_e)$ evaluates as the expression in (\ref{Eq::Lem_Ass}), which is zero by assumption. It follows that $F$ factors through a continuous homomorphism $\widetilde{F} \colon G \to G \ltimes M$ such that the diagrams
	\begin{equation}
		\begin{tikzcd}
			\Gs(v) \arrow[r, hook] \arrow[d, hook] \arrow[rd, "F_v"] & G \arrow[d, "\widetilde{F}"] & \sge \arrow[r, hook] \arrow[d, hook] \arrow[rd, "F_e"] & G \arrow[d, "\widetilde{F}"] \\
			W \arrow[r, "F"]                                         & G \ltimes M                  & W \arrow[r, "F"]                                       & G \ltimes M                 
		\end{tikzcd}
	\end{equation}
	commute whenever $v \in V(\Xi)$ and $e \in E(\Xi)$. Now $f = \operatorname{proj}_2\widetilde{F}$ is a continuous derivation (cf. Remark~\ref{Rem::Sections}) which agrees with $f_v$ and $f_T(t_e)$ for all $v \in V(\Xi)$ and $e \in E(\Xi)$. As the images of $\Gs(v)$ and $t_e$ generate $G$ topologically, $f$ is unique among continuous derivations with this property.
\end{proof}

With this tool at hand, we proceed to prove the profinite analogue of direction $(a) \Rightarrow (b)$ in Theorem~\ref{Thm::Dunwoody}. The proof is a transfiguration of Dunwoody's argument, adjusted substantially to accommodate the eccentricities of the profinite world.

\begin{proposition}
\label{Prop::Partial Subgroup Accessibility}
Let $G$ be a profinite group, $H \leq G$ be a closed subgroup accessible within $G$ via some graph of profinite groups $\gxt$, and $R$ be a profinite ring. For any vertex $w \in V(\Xi)$, there exists a continuous derivation $f \colon G \to R[[G]]$ satisfying $\Gs(w) \cap \Ker(f) = \Gs(w) \cap H$.
\end{proposition}
\begin{proof}
Let $G$ be a profinite group which decomposes as a graph of groups $\PT$, where the spanning tree $\Theta \subseteq \Xi$ is fixed and we shall write $\partial_x \colon \Gs(x) \to G = \PT$ to denote the associated inclusion of profinite groups over a given point $x \in \Xi$. We note that $\partial_x$ is indeed a monomorphism by Proposition \ref{Prop::profinite graph of profinite groups}. Similarly, write \[
\widehat{\partial_x} \colon R[\Gs(x)] \xhookrightarrow{\iota} R[[\Gs(x)]] \xhookrightarrow{\partial^*_x} R[[G]]
\] where $\iota$ is the profinite completion of rings and $\partial_x^*$ is the map induced by $\partial_x$ on completed group algebras.

Assume that the closed subgroup $H \leq G$ is accessible, so that there is a vertex $v \in V(\Xi)$ with $H = \Gs(v)$. If $v = w$ then the zero derivation yields the result; assume thus $v \neq w$. Write $[v,w]$ to denote the minimal subtree of $\Theta$ which contains both $v$ and $w$, and define $e$ to be the unique edge in $[v,w]$ which is incident at $v$. Given any other vertex $\tau \in V(\Xi)$, we shall say that $\tau$ is positive ($\tau > v$) if $[v,\tau]$ contains $e$, neutral if $\tau = v$, and negative ($\tau < v$) if it is neither positive nor neutral. Similarly, an edge $\eta$ shall be positive ($\eta > e$) if both its initial and terminal vertex are positive, neutral ($\eta \approx e$) if exactly one of its initial and terminal vertices is positive, and negative ($\eta < e$) if it is neither positive nor neutral. Note that the only neutral vertex is $v$ but there might be neutral edges other than $e$ if $\Xi$ is not a tree.

We shall define the continuous derivation $f$ separately on all vertex groups and stable letters, and then demonstrate that these definitions glue together appropriately. Given a vertex $\tau \in V(\Xi)$, define the continuous map
\[
f_\tau \colon \Gs(\tau) \xrightarrow{\quad} R[\Gs(\tau)] \xhookrightarrow{\widehat{\partial}_\tau} R[[G]]
\]
as
\[f_\tau(g) = \sum_{k \in \Gs(e)} k(g-1) \]
if $\tau > v$, and
\[f_\tau = 0\]
otherwise. We note that, depending on whether $\tau > v$ or $\tau \leq v$, these maps are precisely the inner derivations associated to the elements $\sum_{k \in \Gs(e)} k \in R[[G]]$ and $0 \in R[[G]]$, respectively, so in particular they form continuous right derivations of $G$ in $R[[G]]$ (cf. \cite[Section 6.8]{RZ}. On the other hand, let $\eta \in E(\Xi)$ be an edge and $t_\eta$ the corresponding stable letter. If $\eta \in \Theta$ then $t_\eta = 1$, and we set $f(t_\eta) = 0$. Otherwise, assume $\eta \in \Xi - \Theta$ and proceed by cases depending on the sign of $\eta$. Set
\[
f(t_\eta) = \sum_{k \in \Gs(e)} k(t_\eta - 1)
\]
if $\eta > e$,
\[
f(t_\eta) = - \sum_{k \in \Gs(e)} k
\]
if $\eta \approx e$, and
\[
f(t_\eta) = 0
\]
if $\eta < 0$. In this way, we have defined continuous derivations $f_\tau$ in $R[[G]]$ on all vertex groups of $\ggx$, and an assignment of stable letters to values in $R[[G]]$. By Proposition~\ref{Prop::Glue}, this data induces a continuous derivation $f \colon G \to R[[G]]$, provided we show that
\begin{equation} \label{Eq::Target derivation}
	\left(f_{d_1(\eta)}(\partial_1(g)) - f(t_\eta)\right)t_\eta^{-1} \partial_0(g)^{-1}t_\eta + f_{d_0(\eta)}(\partial_0(g)^{-1})t_\eta + f(t_\eta) = 0
\end{equation}
whenever $\eta \in E(\Xi)$ and $g \in \Gs(\eta)$. Indeed, choose $\eta \in E(\Xi)$ and write $\alpha = d_0(\eta)$ and $\beta = d_1(\eta)$. We proceed again by cases on $\eta$. If $\eta < e$ then $f(t_\eta) = f_\alpha (g) = f_\beta(g) = 0$ for all $g \in \Gs(\eta)$, so there is nothing to prove. If $\eta > e$ then (\ref{Eq::Target derivation}) evaluates as
\begin{equation}
		\begin{split}
			\sum_{k \in \Gs(e)} k(\partial_1(g) - 1)  t_\eta^{-1} \partial_0(g)^{-1}t_\eta -
			\sum_{k \in \Gs(e)} k(t_\eta - 1)t_\eta^{-1} \partial_0(g)^{-1}t_\eta \\
			+ \sum_{k \in \Gs(e)} k(\partial_0(g)^{-1} - 1)t_\eta + \sum_{k \in \Gs(e)} k(t_\eta - 1)
		\end{split} 
\end{equation}
which we can re-write using the identity $\partial_1(g) = t_\eta^{-1}\partial_0(g)t_\eta$ in $G$, yielding
\begin{equation}
	\begin{split}
		\sum_{k \in \Gs(e)} k \Big(1 - t_\eta^{-1} \partial_0(g)^{-1}t_\eta - \partial_0(g)^{-1}t_\eta + t_\eta^{-1} \partial_0(g)^{-1}t_\eta \\
		+ \partial_0(g)^{-1}t_\eta - t_\eta + t_\eta - 1 \Big)
	\end{split} 
\end{equation}
which reduces to zero. Finally, suppose that $\eta \approx e$ and distinguish further between the cases $\eta = e$ and $\eta \notin \Theta$. In the former case, $t_\eta = 1$ in $G$ and $\partial_0= \partial_1$ on $\Gs(\eta)$. Thus (\ref{Eq::Target derivation}) reduces to
\begin{equation}
	f_\beta(g) - f_\alpha(g) = \pm \sum_{k \in \Gs(e)} k(g - 1) = 0
\end{equation}
where the first equality holds as precisely one of $\alpha, \beta$ is positive and the other is negative, while the second equality holds since $g \in \Gs(\eta) = \Gs(e)$. On the other hand, if $\eta \approx e$ but $\eta \notin \Theta$, then precisely one of $\alpha, \beta$ is positive; assume w.l.o.g. that $\beta$ is positive and $\alpha$ is negative. Then (\ref{Eq::Target derivation}) evaluates as
\begin{equation}
		\sum_{k \in \Gs(e)} k(\partial_1(g) - 1)  t_\eta^{-1} \partial_0(g)^{-1}t_\eta +
		\sum_{k \in \Gs(e)} k t_\eta^{-1} \partial_0(g)^{-1}t_\eta  - \sum_{k \in \Gs(e)} k
\end{equation}
which we can re-write using the identity $\partial_1(g) = t_\eta^{-1}\partial_0(g)t_\eta$ in $G$, yielding
\begin{equation}
	\sum_{k \in \Gs(e)} \left( k - kt_\eta^{-1} \partial_0(g)^{-1}t_\eta + kt_\eta^{-1} \partial_0(g)^{-1}t_\eta - k \right)
\end{equation}
which reduces to zero. We conclude that the collection of continuous derivations $\{f_\tau \colon \Gs(\tau) \to R[[G]]\}_{\tau \in V(\Xi)}$ and the assignment $f : \{t_e\}_{e \in E(\Xi)} \to R[[G]]$ are compatible in the sense of Proposition~\ref{Prop::Glue}, yielding a continuous derivation
\[
f \colon G \to R[[G]]
\]
which extends the assignment and $f_\tau$ for all $\tau \in V(\Xi)$. It remains only to show that $f$ has the postulated properties. Indeed, $f(\Gs(v)) = f_v(\Gs(v)) = 0$, so $\Gs(v) \subseteq \Ker(f)$ and $\Gs(w) \cap \Gs(v) \subseteq \Gs(w) \cap \Ker(f)$. For the opposite of the latter inequality, note that $w > v$ by construction, so
\[
f\at{\Gs(w)}(g) = \sum_{k \in \Gs(e)} k(g-1)
\]
which has a non-zero $g$-coordinate unless $g \in \Gs(e) = \Gs(v) \cap \Gs(w)$. We conclude that $\Ker(f) \cap \Gs(w) = \Gs(v) \cap \Gs(w)$ and the proof is complete.
\end{proof}

Having established Proposition~\ref{Prop::Partial Subgroup Accessibility}, we are now ready to prove the main result of this section, Theorem~\ref{Thm::Subgroup Accessibility}. The idea is to construct a product of derivations, each given by Proposition~\ref{Prop::Partial Subgroup Accessibility}, and then to demonstrate that the kernel of the product is precisely the accessible subgroup. For lack of a normal form theorem for arbitrary elements of a graph of profinite groups, the latter becomes an issue of asymptotics within a profinite group, which we solve using a systematic approximation of elements and their images under the derivation, relative to open subgroups of the domain and codomain whose compatibility we chase.

\MTA*

\begin{proof}
    Assume that $G = \PT$ for some graph of profinite groups $\gxt$ over a finite graph $\Xi$ with spanning tree $\Theta$ such that $\gxt$ has finite edge groups and the closed subgroup $H \leq_c G$ arises as a vertex group $H = \Gs(v)$ for some vertex $v \in V(\Xi)$. Let $\Gamma = \PAT$ be the abstract fundamental group. By Proposition~\ref{Prop::profinite graph of profinite groups}, the profinite group $G$ is isomorphic to the completion of $\Gamma$ with respect to the profinite topology determined by the the neighbourhood basis \[
    \mathcal{U} = \{N \trianglelefteq_f \PAT \mid \forall x\in \Xi, \, N \cap \Gs(x) \trianglelefteq_o \Gs(x)\}
    \] for the identity. Moreover, the latter part of Proposition~\ref{Prop::profinite graph of profinite groups} states that the completion map $\iota \colon \Gamma \to G$ is a monomorphism. Thus, we may identify $\Gamma$ with a dense subgroup of $G$ and shall omit $\iota$ from notation. Given an open normal $U \trianglelefteq_o G$, we shall denote the canonical projection as $\pi_U \colon G \to G/U$, and we will abbreviate $X_U := \pi_U(X)$ and $h_U := \pi_U(h)$ for any $X \subseteq G$ and $h \in G$. Finally, write $\wE = E(\Xi) - E(\Theta)$ for the collection of edges in $\Xi$ with non-trivial stable letters.
	
	When restricted to $\Theta$, the graph of groups $\ggx$ forms a tree of groups $(\Gs, \Theta)$, whose fundamental group $\PP[(\Gs,\Theta,\Theta)]$ generates $\Gamma$ together with the collection of stable letters $\{t_e : e \in \wE\}$. Write $N = \overline{\langle \langle \PP[(\Gs,\Theta,\Theta)] \rangle \rangle}$ for the minimal closed normal subgroup of $G$ containing $\PP[(\Gs,\Theta,\Theta)]$. Furthermore, let $\Phi = \langle t_e : e \in \wE \rangle$ be the subgroup generated by the stable letters, and $\wF$ be its closure within $G$. By \cite[Lemma 3.5]{Zalesskii_Melnikov}, the quotient $q \colon G \to G/N$ restricts to an isomorphism $q\at{\wF} \colon \wF \to G/N$ and $\wF$ is free profinite on the set of stable letters $\{t_e : e \in \wE\}$ associated to edges in $\wE$.\footnote{The cited lemma states that $G/N$ is free profinite on the image of $\{t_e : e \in \wE\}$ under the quotient $q$. But the restriction $q\at{\wF} \colon \wF \to G/N$ admits a surjective section $\sigma \colon G/N \to \wF$ generated by $\sigma \colon q(t_e) \mapsto t_e$, so the composition $\sigma \circ q\at{\wF}$ is also surjective and an isomorphism by \cite[Proposition 2.5.2]{RZ}. In particular, $q\at{\wF} \colon \wF \to G/N$ is injective and must also form an isomorphism.} In particular, this means $\wF = \PP[\hgx]$ for some graph of profinite groups $\hgx$ whose vertex groups are either  $\sge \cong \pZ$ or trivial.\footnote{Let $\Upsilon$ be the graph with vertex set given by the symbols $V(\Upsilon) = \{\tau(e) : \wE\} \sqcup \{\tau_0\}$, edge set $E(\Upsilon) = \wE$ and adjacency maps
	$
		d_0(e) = \tau_0 \text{ and } d_1(e) = \tau(e)
	$
	whenever $e \in \wE$. Define the graph of profinite groups $\hgx$ as
	\[
	\Hs(\tau) = \begin{cases}
		\sge, &\tau = \tau(e) \ \\
		1, &\tau = \tau_0 \\
	\end{cases}
	\]
	whenever $\tau \in V(\Upsilon)$ and with trivial edge groups and edge inclusions. The profinite fundamental group then computes as the free profinite group $\PP[\hgx] = \wF$.}

	We now commence our construction of the postulated derivation as a product of derivations given by Proposition~\ref{Prop::Partial Subgroup Accessibility}. Given any vertex $w \in V(\Xi) - \{v\}$, Proposition~\ref{Prop::Partial Subgroup Accessibility} with graph of groups $\gxt$, accessible subgroup $H = \Gs(v)$ and distinguished vertex $w$ yields a continuous derivation $f_w \colon G \to R[[G]]$ such that $H\subseteq \Ker(f_w)$ and $\Ker(f_w) \cap \Gs(w) = H \cap \Gs(w)$. Similarly, given any edge $e \in \wE$, Proposition~\ref{Prop::Partial Subgroup Accessibility} with graph of groups $\hgx$, accessible subgroup $\Hs(\tau_0) = 1$ and distinguished vertex $\tau(e)$ yields a continuous derivation $f_e' \colon \wF \to R[[\wF]]$ with $\sge \cap \Ker(f_e') = \Hs(\tau(e)) \cap \Hs(\tau_0) = 1$. Set
	\[
	f_e \colon G \xrightarrow{ \, \, q \, \,} \wF \xrightarrow{f_e'} R[[\wF]] \xrightarrow{i^*} R[[G]]
	\]
	whenever $e \in \wE$ is an edge with non-trivial stable letter. We define the continuous derivation
	\[
	f \colon G \to R[[G]]^n
	\]
	as
	\[
	f = \prod_{w \in V(\Xi) - \{v\}} f_w \times \prod_{e \in \wE} f_e
	\]
	where $n = |V(\Xi)| - 1 + |\wE|$. We note  that 
    \begin{equation}
        n = |V(\Xi)| - 1 + |\wE| = (|E(\Theta)| + 1) - 1 + |E(\Xi) - E(\Theta)| = |E(\Xi)|
    \end{equation} 
     is the number of edges in $\Xi$, as postulated. Then
	\begin{equation}\label{Eq::One side of inclusion}
		H \subseteq \bigcap_{w \in V(\Xi) - \{v\}} \Ker(f_w) \cap \bigcap_{e \in \wE} \Ker(f_e) = \Ker(f)
	\end{equation}
	so it remains only to show that the opposite inclusion holds as well. Given a subset $A \subseteq \Xi$ of vertices and edges, we shall write $\Lambda(A) = \langle \Gs(v) : v \in A \cap V(\Xi) \rangle \leq \Gamma$ and $\Sigma(A) = \langle t_e : e \in A \cap E(\Xi) \rangle \leq \Gamma$, as well as
	\[
	\Gamma(A) = \langle \Lambda(A), \Sigma(A) \rangle \leq \Gamma
	\]
	and
	\[
	G(A) = \overline{\Gamma(A)} \leq G
	\]
	the latter of which forms a closed subgroup of $G$. Given any $h \in G$, the collection of subsets $A \subseteq \Xi$ satisfying $h \in G(A)$ forms a finite partially ordered set $\mathcal{M}(A)$ under inclusion, so it must have a minimal element. Fix a choice function
	\[
	\Int \colon G \to \mathcal{P}(\Xi)
	\]
	with $\Int(h)$ minimal in $\mathcal{M}(h)$ for each $h \in G$, and refer to $\Int(h)$ as the \emph{intricacy} of $h$ in $G$. We note that $\Int(h) \subseteq \{v\}$ if and only if $h \in H$.
	
	\textbf{Claim:} For any element $h \in G$ and any open normal subgroup $V \trianglelefteq_o G$, there exists $\delta \in \Gamma$ such that $hV = \delta V$ and $\Int(h) = \Int(\delta)$.
	
	Indeed, write $\mathcal{N}(h)$ for the collection of proper subsets of $\Int(h)$. Then $h \in G(\Int(h))$ but $h \notin G(A)$ for each $A \in \mathcal{N}(h)$ by minimality of $\Int(h)$. As $\Xi$ is finite, so is $\mathcal{N}(h)$, and the set
	\[
	B = hV \, \cap \bigcap_{A \in \mathcal{N}(h)} G - G(A)
	\]
	must be open. Now $h \in G(\Int(h)) \cap B$, so in particular $G(\Int(h)) \cap B$ is non-empty. But $B \cap G(\Int(h)) = B \cap \overline{\Gamma(\Int(h))}$ by definition, so $\Gamma(\Int(h)) \cap B$ must be non-empty as well. It follows that there exists some
	\[
	\delta \in \Gamma(\Int(h)) \,\cap\, hV \,\cap \bigcap_{A \in \mathcal{N}(h)} G - G(A)
	\]
	whence the claim.
	
	We proceed to demonstrate that the opposite inclusion of (\ref{Eq::One side of inclusion}) holds as well. Indeed, suppose for a contradiction that there is $g \in \Ker(f) - H$. Given $A \subsetneqq \Int(g)$, we have $g_A \notin G(A)$ by minimality of $\Int(g)$, so the closure of $G(A)$ implies that there must be $U_A \trianglelefteq_o G$ with $g_AU_A \cap G(A) = \emptyset $. Similarly, as $\ggx$ has finite edge groups, there exists an open normal $U_e \trianglelefteq_o G$ for each $e \in E(\Xi)$ satisfying $\Gs(e) \cap U_e = 1$. As $E(\Xi)$ and $\Int(g) \subseteq \mathcal{P}(\Xi)$ are both finite, the normal subgroup
	\[
	U = \bigcap_{e \in E(\Xi)} U_e \,\,\, \cap \bigcap_{A \subsetneqq \Int(g)} U_A
	\]
	is open in $G$. Consider now the continuous group homomorphism
	\[
	F \colon G \xrightarrow{(\operatorname{id},f)} G \ltimes R[[G]]^n \xrightarrow{(\pi_U, \pi_U^*)} G/U \ltimes R[G/U]^n
	\]
	where $\pi_U^*$ is the induced projection of group algebras. We shall make use of the projections
    $
	\operatorname{proj}_1 \colon G/U \ltimes R[G/U]^n \to G/U
	$
	and
	$
	\operatorname{proj}_2 \colon G/U \ltimes R[G/U]^n \to R[G/U]^n
	$
	onto each factor, which are functions but not necessarily homomorphisms. Let $W = \Ker(F) = U \,\cap\, \Ker(\pi_U^*f)$ which is an open normal subgroup of $G$, as $F$ is a homomorphism and both $R$ and $[G:U]$ are finite. By the universal property of the quotient $\pi \colon G \to G/W$, there is a homomorphism $\widetilde{F}$, and hence a derivation $\widetilde{f}$, such that the diagram
	\begin{equation}
		\begin{tikzcd}
			& R[[G]]^n \arrow[dr, "\pi_U^*"] & \\
			G \arrow[r,"F"] \arrow[dr, "\pi_W"'] \arrow[ru, "f"]& G/U \ltimes R[G/U]^n \arrow[r,"\operatorname{proj}_2"] & R[G/U]^n \\ & G/W \arrow[u,"\widetilde{F}", dashed] \arrow[ur, "\widetilde{f}"', dashed] &
		\end{tikzcd}
	\end{equation}
	commutes. Given a vertex $w \in V(\Xi)$ or an edge $e \in \wE$ with non-trivial stable letter, write $\widetilde{f_w} \colon G/W \to R[G/U]$ or $\widetilde{f_e} \colon G/W \to R[G/U]$ for the projection of $\widetilde{f}$ onto the $w$-factor or $e$-factor, respectively. By the claim, there exists $\gamma \in \Gamma$ of the same intricacy as $g$ satisfying $\gamma W = gW$. Moreover, as $g \notin H$ by assumption, the set $\Int(\gamma) - \{v\}$ is non-empty. In particular, at least one of the sets $\Int(\gamma) \cap V(\Xi) - \{v\}$ and $\Int(\gamma) \cap E(\Xi)$ must be non-empty; we distinguish between these cases.
	
	\textbf{Case 1:} There exists a vertex $w \in \Int(\gamma) \cap V(\Xi) - \{v\}$.
	
	Write $[v,w]$ to denote the minimal subtree of $\Theta$ which contains $v$ and $w$ and let $e \in E([v,w])$ be the unique edge in $[v,w]$ which is incident at $v$. Write $\Lambda = \Gs(w)$ and $\Delta = \Gamma(\Int(\gamma) - \{w\})$, as well as $K = \Gs(e)$, so that $\gamma \in \langle \Lambda, \Delta \rangle$ but $\gamma \notin \Delta$ by minimality of $\Int(\gamma)$. By construction of $U$, we also have $\gamma_U = g_U \notin \Delta_U$. As $\gamma \in \langle \Delta , \Lambda \rangle$, there are tuples $\{\lambda^i : 1 \leq i \leq m\} \subseteq \Lambda$ and $\{\delta^i : 1 \leq i \leq m\} \subseteq \Delta$ such that
	\begin{equation}
		\gamma = \prod_{i=1}^m \delta^i \lambda^i
	\end{equation}
	is in reduced form, i.e. satisfying $\delta^i, \lambda^i \notin \Lambda \cap \Delta$ for all $i$ except possibly $\lambda^1$ and $\delta^n$. Given $1 \leq i \leq m$, write $\gamma^i = \prod_{j=i}^m \delta^j\lambda^j$ and $\gamma^{m+1} = 1$ for convenience. Moreover, let $\varepsilon_\Delta \colon R[G/U] \to R$ be the map of $R$-modules generated by $\varepsilon_\Delta(\delta_U) = 0$ for $\delta_U \in \Delta_U$ and $\varepsilon_\Delta(x_U) = 1$ for $x_U \in G/U-\Delta_U$. Then
	\begin{align}
		\varepsilon_\Delta \widetilde{f_w} (g_W) &= \varepsilon_\Delta \pi^*_U f_w \left(\gamma\right) \\
		&= \varepsilon_\Delta \left[\sum_{i=1}^m \sum_{\kappa \in K} \kappa_U(\lambda_U^i-1) \cdot \gamma^{i+1}_U \right] \label{Eq::Huge comp 1}\\
		&= \varepsilon_\Delta \left[ \sum_{\kappa \in K} \kappa_U\left((\delta^1_U)^{-1}\gamma_U - \sum_{i=2}^m(1 - (\delta_U^i)^{-1}) \cdot \gamma_U^i- 1 \right) \right] \label{Eq::Huge comp 2}\\
		&= \sum_{\kappa \in K} \left[1 - \varepsilon_\Delta \left(\sum_{i=2}^m(1 - (\delta_U^i)^{-1}) \cdot \gamma_U^i \right) \right] \label{Eq::Huge comp 3} \\
		&= |K| \neq 0 \label{Eq::Huge comp 4}
	\end{align}
	where Equation (\ref{Eq::Huge comp 1}) is obtained using the derivation law and the explicit definition of $f$ given in the proof of Proposition~\ref{Prop::Partial Subgroup Accessibility}, Equation (\ref{Eq::Huge comp 3}) follows from the equality $\varepsilon_\Delta(\kappa_U (\delta^1_U)^{-1} \gamma_U) = \varepsilon_\Delta(\gamma_U) = 1$, while (\ref{Eq::Huge comp 4}) follows from the conjunction of the equality $K \cap U = 1$, which holds by construction of $U$, and the inequality $|K| \cdot 1 \neq 0$, which holds in $R$ by assumption. It follows that
	\begin{equation}
		\pi_U^*f_w(g) = \widetilde{f_w} \pi_W (g) = \widetilde{f_w}(g_W) \neq 0
	\end{equation}
	so in particular $f(g) \neq 0$, contradicting the assumption that $g \in \Ker(f)$.
	
	\textbf{Case 2:} There are only edges in $\Int(\gamma) - \{v\}$.
	
	In that case, we may assume that $\Int(\gamma) - \{v\} \subseteq \wE = E(\Xi) - E(\Theta)$, as all stable letters associated to edges in $E(\Theta)$ are trivial. Moreover, $\Int(\gamma) \cap \wE$ is non-empty by the assumption that $g \notin H$. Choose any $e \in \wE$ and write $\Lambda = \langle t_e \rangle \leq \Gamma$ as well as $\Delta = \Gamma(\Int(\gamma) - \{e\})$, so that $\gamma \in \langle \Lambda, \Delta \rangle$ but $\gamma \notin \Delta$. Proceeding exactly as in Case 1 with the definition of $f_e$ in terms of the graph of profinite groups $\hgx$ and $K = 1$, we find that
	\begin{equation}
		\pi_U^*f_e(g) = \widetilde{f_e} \pi_W (g) = \widetilde{f_e}(g_W) \neq 0
	\end{equation}
	so in particular $f(g) \neq 0$, contradicting the assumption that $g \in \Ker(f)$. We conclude that $\Ker(f) = H$ and the proof is complete.
\end{proof}
\begin{corollary}\label{Cor::Amalgam}
    Let $G = H \amalg_K L$ be a profinite free product with amalgamation along a finite group $K$. For any finite ring $R$ whose characteristic does not divide the order of $K$ there exists a continuous derivation $f \colon G \to  R[[G]]$ with $\Ker(f) = H$.
\end{corollary}

\section{Recognition of Amalgamated Factors} \label{Sec::Amalgams}
In this section, we prove Theorem~\ref{Thm::Recognition}. As an entre\'e, we commence with a few lemmata whereto we outsource the most pungent Bass--Serre theoretic ingredients of the stew which constitutes the proof of Theorem~\ref{Thm::Recognition}.

\begin{lemma}\label{Lem::Bass--Serre Magic 1}
    Let $\gxt$ be a graph of groups over a finite graph $\Xi$ with finite edge groups whose fundamental group $\Gamma = \PAT$ is finitely generated virtually free. For every vertex $v \in V(\Xi)$, there exists a graph of groups $\hys$ over a finite graph $\Upsilon$ such that $\Gamma = \PA[\hys]$, there exists a vertex $u \in \Upsilon$ with $\Hs(u) = \Gs(v)$ and $\Hs(w)$ is finite whenever $w \in V(\Upsilon) - \{u\}$.
\end{lemma}
\begin{proof}
    Having chosen a spanning tree $\Theta$ of $\Xi$, we identify each $\Gs(x)$ for $x \in \Xi$ with its image in $\Gamma$. Choose a vertex $v \in V(\Xi)$ and write $\Delta = \Gs(v)$. We proceed by induction on the number of vertices $w \in V(X) - \{v\}$ whose vertex group $\Gs(w)$ has infinite cardinality. For the base case $n = 0$, there is nothing to prove. For the inductive step, assume that the result holds for $n$ and that $\gxt$ has $n+1$ vertex groups with infinite cardinality. Choose any $w \in V(X) - \{v\}$ whose vertex group $\Gs(w)$ has infinite cardinality $\Lambda := \Gs(w)$.
   
   The subgroup $\Lambda \leq \Gamma$ is virtually free and finitely generated: the latter follows via \cite[Proposition 2.13]{Bieri_Book} from the assumption that the fundamental group $\Gamma = \PA$ and all edge groups $\Gs(e)$ for $e \in E(\Xi)$ are finitely generated. Hence Theorem~\ref{Thm::KPS} yields a graph of groups $\lxt$, over a finite graph $\Zeta$ whose vertex groups are finite and which satisfies $\Lambda = \PA[\lxt]$. Let $\Psi$ be the associated structure tree. By assumption, the edge groups of $\gxt$ are finite, so any such edge group $\Gs(\eta)$ for $\eta \in E(\Xi)$ incident at $w$ must fix a vertex in the induced action of $\Gs(e) \leq \Lambda$ on the tree $\Psi$. It follows that for each edge $\eta \in E(\Xi)$ incident at $w$, there exists a vertex $\tau(\eta) \in V(\Zeta)$ such that the inclusion $\Gs(\eta) \to \Lambda$ factors as
   \[
   \Gs(\eta) \to l_\eta^{-1} \Ls(\tau(\eta)) l_\eta \to \Lambda
   \]
   where $l_\eta \in \Lambda$ is some conjugating factor. Thus, after possibly conjugating vertex groups in $\gxt$, we may remove the vertex $w$ and replace it with an embedded copy of the graph of groups $\lxt$, yielding a graph of groups $\hys$ over the graph $\Upsilon = (\Xi - \{w\}) \cup \Zeta$ which satisfies $\Gamma = \PA[\hys]$ and such that $\Delta$ is conjugate to $\Gs(u)$ for some vertex $u \in V(\Upsilon)$. The situation is illustrated in in Figure~\ref{Fig::BM1}; we refer the reader to \cite[Lemma 4.12]{Guirardiel_Levitt} for details of such a construction.\input{figure1} This new graph of groups decomposition $\Gamma = \PA[\hys]$ may be over a larger graph, but it has at most $n$ vertex groups of infinite cardinality. After possibly conjugating back the vertex groups of $\hys$, one obtains a graph of groups decomposition $\Gamma = \PA[(\Hs', \Upsilon, \Sigma)]$ with $\Delta = \Hs'(u)$. The result now follows by induction.
\end{proof}

\begin{lemma}\label{Lem::K conjugate to Chi}
    Let $\Delta$ be a finitely generated virtually free group and $\widehat{\Delta}$ its profinite completion. Any finite subgroup $K \leq \widehat{\Delta}$ is conjugate to some $\Chi \leq \Delta$ in $\widehat{\Delta}$.
\end{lemma}
\begin{proof}
    By Theorem~\ref{Thm::KPS}, the finitely generated virtually free group $\Delta$ decomposes as a graph of finite groups $\Delta = \PAT$ over a finite graph $\Xi$. It follows via Corollary \ref{Cor::profinite graph of finite groups} that the profinite completion $\widehat{\Delta}$ decomposes in the profinite category as an injective graph of finite groups $\Delta = \PAT$ over the finite graph $\Xi$. An application of \cite[Theorem 7.1.2]{Ribes_Graphs} then yields the result.
\end{proof}

We shall apply Lemma~\ref{Lem::K conjugate to Chi} in conjunction with the following result, which follows from an argument analogous to that of Lemma~\ref{Lem::Bass--Serre Magic 1}. Throughout the remainder of this section, we shall use the notation $(-)^g$ to denote the conjugation automorphism $g^{-1}(-)g$ associated to an element $g$ in a group $\Gamma$.

\begin{lemma}\label{Lem::Bass--Serre--2}
    Let $\Gamma = \PAT$ be the fundamental group of a graph of groups over a finite graph $\Xi$ with finite edge groups and let $\Delta = \Gs(v) \leq \Gamma$ be the image of a vertex group for some $v \in V(\Xi)$. Assume that there exists a finite subgroup $\Chi \leq \Delta$ such that for each edge $e \in E(\Xi)$ incident at $v$, the edge group $\Gs(e)$ is conjugate in $\Delta$ to a subgroup of $\Chi$. Then there exists $\Lambda \leq \Gamma$ such that $\Gamma = \Delta \am \Lambda$.
\end{lemma}
\begin{proof}
    We shall construct a graph of groups $\hys$ from $\gxt$ by blowing up the vertex $v$ into the amalgam $\Delta \am \Chi$, akin to the proof of \cite[Proposition 2.2]{Guirardiel_Levitt}. Indeed, let $\Upsilon = \Xi \cup \{e,w\}$ be the graph obtained from $\Xi$ by adjoining a fictitious edge $e$ and a fictitious vertex $w$ with $d_0(e) = v$ and $d_1(e) = w$. Fix the spanning tree $\Sigma = \Theta \cup \{e\}$ for $\Upsilon$. Given any point $x \in \Xi - \{v\}$, denote by $\eta(x)$ the unique edge incident at $v$ in the minimal subtree $[x, v]$ of $\Sigma$ containing $v$ and $x$. By assumption, for each edge $\eta$ incident at $v$ there exists $\delta(\eta) \in \Delta$ such that $\Gs(\eta)^{\delta(\eta)} \leq \Chi$. Consider now the graph of groups $(\Hs, \Upsilon)$ given by
    \[
    \Hs(x) = \begin{cases}
        \Delta, &x = w\\
        \Chi, &x = e\\
        \Gs(x)^{\delta(\eta(x))}, & x \in \Xi
    \end{cases}
    \]
    for $x \in \Upsilon$, and whose inclusions are the appropriate conjugates of inclusions in $\ggx$, the postulated inclusions $\Gs(\eta)^{\delta(\eta)} \to \Chi$, as well as the natural inclusions $\Chi \to \Delta$ and $\Chi \to \Chi$. The situation is illustrated in Figure~\ref{Fig::HGX from GGX}; one verifies that this forms a well-defined graph of groups and that $\Gamma = \PA[\hys]$.\input{figure2}The fundamental group $\Lambda = \PA[(\Hs, \Xi, \Theta)]$ of the graph of groups $\Hs$ restricted to $\Xi \subseteq \Upsilon$ with spanning tree $\Theta$ then gives $\Gamma = \Delta \am \Lambda$, as postulated.
\end{proof}
Finally, we prove the main theorem of the present work, Theorem~\ref{Thm::Recognition}. The idea is to pass from a profinite splitting to a continuous derivation using Theorem~\ref{Thm::Subgroup Accessibility}, restrict the derivation to a projective abstract module using Theorem~\ref{Thm::Homological Algebra}, and convert the latter to an abstract graph of groups using Theorem~\ref{Thm::Dicks_Dunwoody}. To mould the obtained abstract graph of groups into the required form, we appeal directly to the profinite structure tree associated to the initial profinite splitting, as well as the abstract Bass--Serre theoretic manipulations established in the present section.

\MTB*

\begin{proof}
The direction $(b) \Rightarrow (a)$ holds in much greater generality: if $\Gamma = \Delta \am \Lambda$ is residually finite and $\Chi^h = K$ is finite then $\Gh = \Db \amalg_{\upchi} \overline{\Lambda} = \Db \amalg_K \overline{\Lambda}^h$, where the first equality holds by Proposition~\ref{Prop::completion of fundamental group} and the second equality holds as conjugation by $h \in \Db$ is an automorphism of $\Gh$ which fixes $\Db$.
 
For the direction $(a) \Rightarrow (b)$, assume that $\Delta \leq \Gamma$ are both finitely generated virtually free groups and that the profinite completion of $\Gamma$ is a profinite free product  $\Gh = \Db \amalg_K L$ with amalgamation along a finite subgroup $K \leq \Gh$. Let $T$ be the profinite structure tree associated to this splitting, and write $\alpha_\Delta \in V(T)$ to denote the vertex associated to the coset $\Db$ in $T$. As noted in Lemma~\ref{Lem::Induced Topology}, $\Delta$ is closed in the profinite topology on $\Gamma$, and so $\Db \cap \Gamma = \Delta$. Let $p$ be a prime number which is strictly greater than the order of $K$ and consider $\Fp[[\Gh]]$ as a profinite right $\Gh$-module. By Corollary \ref{Cor::Amalgam}, there exists a continuous derivation
    \[
    f \colon \Gh \to \Fp[[\Gh]]
    \]
satisfying $\Ker(f) = \Db$. Consider now $\Fp[[\Gh]]$ as an $\Fp[\Gamma]$-module, which is isomorphic to a direct union of free $\Fp[\Gamma]$-modules by Theorem~\ref{Thm::Homological Algebra}. As $\Gamma$ is finitely generated, the restriction $f\at{\Gamma} \colon \Gamma \to \Fp[[\Gh]]$ is a derivation with finitely generated image, so this image must in fact lie in some factor of the direct union, and there exists a projective right $\Fp[\Gamma]$-module $P$ such that
\[
	f\at{\Gamma} \colon \Gamma \to P
\]
forms a derivation\footnote{Technically, Theorem~\ref{Thm::Homological Algebra} deals with left modules, while Corollary \ref{Cor::Amalgam} and Theorem~\ref{Thm::Dunwoody} are concerned with right derivations. However, right derivations in a right module are equivalent to left derivations in a left module over the opposite group (cf. Section \ref{Sec::Prelim:Juxta}). Thus $f$ corresponds to a continuous left derivation $f^{\mathrm{op}} \colon \Gh \to \Fp[[\Gh]]$ in the left $\Gh^{\mathrm{op}}$-module $\Fp[[\Gh]]$. But $\Fp[[\Gh]]$ is isomorphic to a direct union of free left $\Fp[\Gamma^{\mathrm{op}}]$-modules and the restriction of $f^{\mathrm{op}}$ to $\Gamma^{\mathrm{op}}$ forms a left derivation $f^{\mathrm{op}} \at{\Gamma^{\mathrm{op}}} \colon \Gamma^{\mathrm{op}} \to P$
in some projective left $\Fp[\Gamma^{\mathrm{op}}]$-submodule $P$ of $\Fp[[\Gh]]$, which satisfies $\Ker(f^{\mathrm{op}} \at{\Gamma^{\mathrm{op}}}) = \overline{\Delta^{\mathrm{op}}} \cap \Gamma^{\mathrm{op}} = \Delta^{\mathrm{op}}$. Now the latter is equivalent to the existence of a right derivation $f\at{\Gamma} \colon \Gamma \to P$ in the projective right $\Fp[\Gamma]$-module $P$, which satisfies $\Ker(f\at{\Gamma}) = \Delta$.} with $\Ker(f\at{\Gamma}) = \Delta$. By Theorem~\ref{Thm::Dicks_Dunwoody}, it follows that there exists a graph of groups $\ggx$ over a finite graph $\Xi$ with finite edge groups such that
\[
	\Gamma = \PA
\]
and $\Gs(v) = \Ker(f\at{\Gamma}) = \Delta$ for some vertex $v \in V(\Xi)$. Using Lemma~\ref{Lem::Bass--Serre Magic 1}, we may reduce to the case where $\Gs(x)$ is finite whenever $x \in \Xi - \{v\}$. We fix a spanning tree $\Theta$ for $\Xi$; whenever $x \in \Xi$, we shall identify $\Gs(x)$ with its image $\Gs(x) \leq \Gamma \cong \PAT$ under the canonical inclusion. After possibly replacing stable letters with their inverses, we may assume that each edge $\eta \in E(\Xi)$ incident at $v$ starts at $v$, i.e. $d_0(\eta) = v$ whenever $v \in \{d_0(\eta), d_1(\eta)\}$. Moreover, after shrinking superfluous edges, we may assume that there is no edge $\eta$ incident at $v$ whose other endpoint $\tau$ satisfies $\Gs(\tau)^{t_\eta} \leq \Gs(v) = \Delta$. By Lemma~\ref{Lem::Induced Topology}, there is an isomorphism $\Db \cong \widehat{\Delta}$ which fixes $\Delta = \Gs(v)$. It follows using Lemma~\ref{Lem::K conjugate to Chi} that there is a finite subgroup $\Chi$ of the abstract group $\Delta$ which is conjugate to $K$ in $\Db$, i.e. there exists $h \in \Db$ such that $\Chi = K^h$ holds.

In light of Lemma~\ref{Lem::Bass--Serre--2}, it will suffice now to show that for each edge $\eta \in E(\Xi)$ which is incident at $v$, there exists $\delta(\eta) \in \Delta$ such that $\Gs(\eta)^{\delta(\eta)} \leq \Chi$. In fact, it will suffice to show that there exists $y(\eta) \in \Db$ with this property: the existence of $\delta(\eta) \in \Delta$ which conjugates $\Gs(\eta)$ to a subgroup of $\Chi$ then follows from the fact that virtually free groups are subgroup conjugacy separable (see \cite[Corollary 1.2]{Chagas_Zalesskii}) and $\Db \cong \widehat{\Delta}$ (see Lemma~\ref{Lem::Induced Topology}). Choose any $\eta \in E(\Xi)$ incident at $v$ and write $\tau$ for its other endpoint, so that $d_0(\eta) = v$ and $d_1(\eta) = \tau$ by the orientation chosen above; write also $t = t_\eta$ for simplicity. If $\tau = v$, then $\Gs(\eta) \leq \Delta \cap \Delta^t$, and $t \notin \Delta$ by \cite[Theorem IV.2.1]{lyndon_schupp}. In this case, there exists $g \in \Db$ with $\Gs(\eta) \leq K^g$ by \cite[Corollary 7.1.5(b)]{Ribes_Graphs}, and $y(\eta) = gh$ yields the claim. Assume hence $\tau \neq v$, so in particular $\Gs(\tau)$ must be finite. It follows by Proposition~\ref{Prop::Finite Group Profinite Tree} that the action of $\Gs(\tau)$ on the profinite structure tree $T$ associated to the splitting $\Gh = \Db \amalg_K L$ must fix a vertex $\alpha \in V(T)$. If $\alpha = t \alpha_\Delta$ then $\Gs(\tau)^{t} \leq \Stab[\Gh]{\alpha_\Delta} \cap \Gamma = \Db \cap \Gamma = \Delta$, contradicting the assumption that no vertex group adjacent to $v$ maps to a subgroup of $\Gs(v) = \Delta$ when conjugated by the stable letter of a connecting edge. Thus, we may assume further that $\alpha \neq t \alpha_\Delta$ and the minimal subtree $[t^{-1}\alpha, \alpha_\Delta]$ of $T$ containing $t^{-1}\alpha$ and $\alpha_\Delta$ must not be a singleton. Moreover, following Remark \ref{Rem::Closed Edge Set}, the edge set $E(T)$ is closed in $T$, so $E([t^{-1}\alpha,\alpha_\Delta]) = E(T) \cap [t^{-1}\alpha,\alpha_\Delta]$ must be closed in $[t^{-1}\alpha,\alpha_\Delta]$. It follows by \cite[Proposition 2.1.6(c)]{Ribes_Graphs} that there is an edge $\widetilde{\eta} \in [t^{-1}\alpha,\alpha_\Delta]$ which is incident at $\alpha_\Delta$. Then $\widetilde{\eta}$ is the edge in $T$ associated to the coset $gK$ for some $g \in \Db$, and we obtain
	\begin{equation}
		\Gs(\eta) \leq \Gs(v) \cap \Gs(\tau)^t \leq \Stab[\Gh]{\alpha_\Delta} \cap \Stab[\Gh]{t^{-1}\alpha} \leq \Stab[\Gh]{\widetilde{\eta}} = K^{g^{-1}}
	\end{equation}
	where the third inclusion derives from \cite[Corollary 4.1.6]{Ribes_Graphs}. Thus, $y(\eta) = gh$ is an element of $\Db$ which conjugates $\Gs(\eta)$ into $\Chi$. Via the subgroup conjugacy separability of $\Delta$, one now obtains $\delta(\eta) \in \Delta$ with $\Gs(\eta)^{\delta(\eta)} \leq \Chi$. An application of Lemma~\ref{Lem::Bass--Serre--2} then yields $\Gamma = \Delta \am \Lambda$ for some $\Lambda \leq \Gamma$ and $\Chi$ is conjugate to $K$ via $h \in \Db$. We conclude that condition $(b)$ holds and the proof is complete. 
\end{proof}

\bigskip 
\printbibliography
\end{document}

%% file: figure1.tex
\begin{figure}
	\captionsetup{width=.85\textwidth}
	\centering
	\vspace{-1.5cm}
	\begin{subfigure}{1\textwidth}
		\captionsetup{width=.85\textwidth}
		\centering
			\begin{tikzpicture}
					\pgfplotsset{
					width=1\textwidth
				}
			\begin{axis}[axis equal, axis lines=none,view={15}{20}]
				\addplot3 [only marks] coordinates {
					(1,0,0)
					(2,1.7321,0)
					(2,-1.7321,0)
					
					(-1,0,0)
					(-2,1.7321,0)
					(-2,-1.7321,0)
					
					(0.25,0,1.5)
					(1.375,0.6495,1.5)
					(1.375,-0.6495,1.5)
				};
				
				\addplot3 [no marks] coordinates {(1,0,0)(-1,0,0)};
				\draw[densely dotted] (axis cs: 1,0,1.5) circle[radius = {transformdirectionx(1)}];
				
				\addplot3 [no marks] coordinates {(1,0,0)(2,1.7321,0)};
				\addplot3 [no marks] coordinates {(1,0,0)(2,-1.7321,0)};
				
				\addplot3 [no marks] coordinates {(2,1.7321,0) (3,1.7321,0)};
				\addplot3 [no marks] coordinates {(2,1.7321,0) (1.4226,2.7321,0)};
				
				\addplot3 [no marks] coordinates {(2,-1.7321,0) (3,-1.7321,0)};
				\addplot3 [no marks] coordinates {(2,-1.7321,0) (1.4226,-2.7321,0)};
				
				
				\addplot3 [no marks] coordinates {(-1,0,0)(-2,1.7321,0)};
				\addplot3 [no marks] coordinates {(-1,0,0)(-2,-1.7321,0)};
				
				\addplot3 [no marks] coordinates {(-2,1.7321,0) (-3,1.7321,0)};
				\addplot3 [no marks] coordinates {(-2,1.7321,0) (-1.4226,2.7321,0)};
				
				\addplot3 [no marks] coordinates {(-2,1.7321,0) (-2,-1.7321,0)};
				
				\addplot3 [no marks] coordinates {(-2,-1.7321,0) (-3,-1.7321,0)};
				\addplot3 [no marks] coordinates {(-2,-1.7321,0) (-1.4226,-2.7321,0)};

				\addplot3 [no marks, densely dotted] coordinates {(1,0,0) (2,0,1.5)};
				\addplot3 [no marks, densely dotted] coordinates {(1,0,0) (0,0,1.5)};
				
				\addplot3 [no marks] coordinates {(0.25,0,1.5)(1.375,0.6495,1.5)};
				\addplot3 [no marks] coordinates {(0.25,0,1.5)(1.375,-0.6495,1.5)};
				\addplot3 [no marks] coordinates {(1.375,0.6495,1.5)(1.375,-0.6495,1.5)};

				\node [below left] at (axis cs: 1,0,0) { $\Lambda$};
				\node [below right] at (axis cs: -1,0,0) { $\Delta$};
				\node[above] at (axis cs: 1,1,1.5) { $(\Ls, \Zeta, \Zeta_0)$};

				\node [below right ] at (axis cs: 2,1.7321,0){\vertexs[j]};
				\node [below right ] at (axis cs: 2,-1.7321,0) {\vertexs[k]};
				
				\node [above left ] at (axis cs: -2, 1.7321,0) {\vertexs[ij]};
				\node [below left ] at (axis cs: -2, -1.7321,0) {\vertexs[ik]};
				
				\node [above ] at (axis cs: 0,0,0) {\edge[i]};
				\node [below right = -3pt] at (axis cs: 1.5,0.8660,0) {\edge[j]};
				\node [below left = -3pt] at (axis cs: 1.5,-0.8660,0) {\edge[k]};
			\end{axis}
		\end{tikzpicture}
			\vspace{-1cm}
	\caption{The graph of groups $(\Gs, \Xi, \Theta)$ with $\Gs(v) = \Delta$ and $\Lambda = \PA[(\Ls, \Zeta, \Zeta_0)]$. Edges incident at $w$ fix a vertex in the structure tree $\Psi$ of $(\Ls, \Zeta, \Zeta_0)$.}
\end{subfigure}

\begin{subfigure}{\textwidth}
	\captionsetup{width=.85\textwidth}
\centering
	\vspace{-1.5cm}
\begin{tikzpicture}
		\pgfplotsset{
		width=1\textwidth
	}
	\begin{axis}[axis equal, axis lines=none,view={10}{20}]
		\addplot3 [only marks] coordinates {
			
			(2,1.7321,0)
			(2,-1.7321,0)
			
			(-1,0,0)
			(-2,1.7321,0)
			(-2,-1.7321,0)
			
			(0.5,0,0)
			(1.25,0.4330,0)
			(1.25,-0.4330,0)
		};
		
		\addplot3 [no marks] coordinates {(0.5,0,0)(-1,0,0)};
		
		\addplot3 [no marks] coordinates {(1.25,0.4330,0)(2,1.7321,0)};
		\addplot3 [no marks] coordinates {(1.25,-0.4330,0)(2,-1.7321,0)};
		
		\addplot3 [no marks] coordinates {(2,1.7321,0) (3,1.7321,0)};
		\addplot3 [no marks] coordinates {(2,1.7321,0) (1.4226,2.7321,0)};
		
		\addplot3 [no marks] coordinates {(2,-1.7321,0) (3,-1.7321,0)};
		\addplot3 [no marks] coordinates {(2,-1.7321,0) (1.4226,-2.7321,0)};
		
		
		\addplot3 [no marks] coordinates {(-1,0,0)(-2,1.7321,0)};
		\addplot3 [no marks] coordinates {(-1,0,0)(-2,-1.7321,0)};
		
		\addplot3 [no marks] coordinates {(-2,1.7321,0) (-3,1.7321,0)};
		\addplot3 [no marks] coordinates {(-2,1.7321,0) (-1.4226,2.7321,0)};
		
		\addplot3 [no marks] coordinates {(-2,1.7321,0) (-2,-1.7321,0)};
		
		\addplot3 [no marks] coordinates {(-2,-1.7321,0) (-3,-1.7321,0)};
		\addplot3 [no marks] coordinates {(-2,-1.7321,0) (-1.4226,-2.7321,0)};

		\addplot3 [no marks] coordinates {(0.5,0,0)(1.25,0.4330,0)};
		\addplot3 [no marks] coordinates {(0.5,0,0)(1.25,-0.4330,0)};
		\addplot3 [no marks] coordinates {(1.25,0.4330,0)(1.25,-0.4330,0)};
		
		\draw[densely dotted] (axis cs: 1,0,0) circle[radius = {transformdirectionx(1.35)}];
		\node [below] at (axis cs: -1,0,0) { $ \qquad l_{\eta_i}^{-1}\Delta l_{\eta_i}$};
		
		\node [below right] at (axis cs: 2,1.7321,0){\vertexpowers[j]{j}};
		\node [below right] at (axis cs: 2,-1.7321,0) {\vertexpowers[k]{k}};
		
		\node [above left] at (axis cs: -2, 1.7321,0) {\vertexpowers[ij]{i}};
		\node [below left = -2pt] at (axis cs: -2, -1.7321,0) {\vertexpowers[ik]{i}};
		
		\node [above] at (axis cs: 0.5, 0,0) {$\tau(\eta_i)$};
		\node [below right = -2pt] at (axis cs: 1.25,0.4330,0) {$\tau(\eta_j)$};
		\node [below left = -2pt] at (axis cs: 1.25,-0.4330,0) {$\tau(\eta_k)$};
		
	\end{axis}
\end{tikzpicture}
\vspace{-2cm}
\caption{The graph of groups $\hys$ obtained by embedding $\lxt$ in place of the vertex $w$ in a conjugated copy of $\gxt$.}
\end{subfigure}
\caption{Constructing $\hys$ from $\gxt$.}
\label{Fig::BM1}
\end{figure}

%% file: figure2.tex
\begin{figure}
\vspace{-2.5cm}
\begin{tikzpicture}
    \pgfplotsset{
        width=1\textwidth
    }
    \begin{axis}[axis equal, axis lines=none,view={15}{20}]
        \addplot3 [only marks] coordinates {
            
            (1,0,1.5)
            
            (1,0,0)
            (2,1.7321,0)
            
            (2,-1.7321,0)
            
            (-1,0,0)
            (-2,1.7321,0)

            (-2,-1.7321,0)
        };
        
        \addplot3 [no marks, densely dashed] coordinates {(1,0,0)(1,0,1.5)};
        \addplot3 [no marks] coordinates {(1,0,0)(-1,0,0)};
        
        \addplot3 [no marks] coordinates {(1,0,0)(2,1.7321,0)};
        \addplot3 [no marks] coordinates {(1,0,0)(2,-1.7321,0)};
        
        \addplot3 [no marks] coordinates {(2,1.7321,0) (3,1.7321,0)};
        \addplot3 [no marks] coordinates {(2,1.7321,0) (1.4226,2.7321,0)};
        
        \addplot3 [no marks] coordinates {(2,-1.7321,0) (3,-1.7321,0)};
        \addplot3 [no marks] coordinates {(2,-1.7321,0) (1.4226,-2.7321,0)};
        
        
        \addplot3 [no marks] coordinates {(-1,0,0)(-2,1.7321,0)};
        \addplot3 [no marks] coordinates {(-1,0,0)(-2,-1.7321,0)};
        
        \addplot3 [no marks] coordinates {(-2,1.7321,0) (-3,1.7321,0)};
        \addplot3 [no marks] coordinates {(-2,1.7321,0) (-1.4226,2.7321,0)};
        
        \addplot3 [no marks] coordinates {(-2,-1.7321,0) (-3,-1.7321,0)};
        \addplot3 [no marks] coordinates {(-2,-1.7321,0) (-1.4226,-2.7321,0)};
        
        \addplot3 [no marks] coordinates {(-2,1.7321,0) (-2,-1.7321,0)};

        \node [above left] at (axis cs: 1,0,0) { $\Chi$};
        \node [below left] at (axis cs: 1,0,1.5) { $\Delta$};
        \node [below] at (axis cs: -1,0,0) {$\qquad$ \vertexpower[i]{i}};
        \node [below right] at (axis cs: 2,1.7321,0){\vertexpower[j]{j}};
        \node [below right] at (axis cs: 2,-1.7321,0) {\vertexpower[k]{k}};
        
        \node [above left] at (axis cs: -2, 1.7321,0) { \vertexpower[ij]{i}};
        \node [below left] at (axis cs: -2, -1.7321,0) {\vertexpower[ik]{i}};   
    \end{axis}
\end{tikzpicture}
\vspace{-2cm}
	\caption{Obtaining $(\mathcal{H}, \Upsilon, \widetilde{\Theta})$ from $(\Gs, \Xi, \Theta)$. The conjugation factors $\delta(\tau)$ depend only on the connected component of $\Xi - \{v\}$.}
	\label{Fig::HGX from GGX}
\end{figure}